\documentclass{amsart}

\usepackage{amsmath}
\usepackage{amssymb}
\usepackage{amscd}

\usepackage[all]{xy}
\usepackage{url}
\usepackage[shortlabels]{enumitem}
\usepackage{tensor}
\usepackage{xcolor}

\usepackage{hyperref}

\numberwithin{equation}{section}

\theoremstyle{plain}
\newtheorem{theorem}{Theorem}[section]
\newtheorem{corollary}[theorem]{Corollary}
\newtheorem{lemma}[theorem]{Lemma}
\newtheorem{proposition}[theorem]{Proposition}

\theoremstyle{definition}
\newtheorem{definition}[theorem]{Definition}

\newtheorem{cond}[theorem]{Condition}

\theoremstyle{remark}
\newtheorem{remark}[theorem]{Remark}

\theoremstyle{definition}

\setlist[enumerate,1]{label=(\roman*)}
\setlist[description]{font=\normalfont\emph}

\newcommand{\ZZ}{\mathbb{Z}}
\newcommand{\QQ}{\mathbb{Q}}

\newcommand{\RR}{\mathbb{R}}

\newcommand{\PP}{\mathbb{P}}

\renewcommand{\AA}{\mathbb{A}}

\DeclareMathOperator{\GL}{GL}
\DeclareMathOperator{\PGL}{PGL}

\DeclareMathOperator{\Gal}{Gal}
\DeclareMathOperator{\Spec}{Spec}

\DeclareMathOperator{\Proj}{Proj}

\DeclareMathOperator{\Aut}{Aut}

\newcommand{\m}{\mathfrak{m}}
\newcommand{\p}{\mathfrak{p}}

\newcommand{\OO}{\mathcal{O}}

\newcommand{\X}{\mathcal{X}}

\newcommand{\U}{\mathcal{U}}

\newcommand{\Sc}{\mathcal{S}}
\renewcommand{\P}{\mathcal{P}}

\newcommand{\Xb}{\bar{X}}

\newcommand{\Ub}{\bar{U}}

\newcommand{\rig}{\mathrm{rig}}

\newcommand{\et}{\mathrm{et}}

\newcommand{\inj}{\hookrightarrow}

\newcommand{\iso}{\stackrel{\sim}{\to}}

\newcommand{\lexp}[2]{\tensor[^#1]{#2}{}}

\begin{document}

\title{Explicit local stable resolution of cusps}

\author{Stefan Wewers}
\address{Institute for Algebra and Number Theory, Ulm University}
\email{stefan.wewers@uni-ulm.de}



\maketitle

\begin{abstract}
	This article is a companion to \cite{SSW}. We complete the local step left open
	there by giving an explicit stable resolution of cusps on GIT-stable plane
	quartic models. More generally, for a smoothing of an ordinary plane cusp over a
	complete discretely valued field, we show that after a finite separable extension
	and a suitable choice of coordinates, a single weighted blow-up with weights
	\((1,2,3)\) gives the stable resolution. The exceptional component is an explicit
	semistable Weierstrass cubic. The construction is effective and is implemented
	in the Sage package \texttt{StabilityFunction}.
\end{abstract}

\section{Introduction} \label{sec:intro}

This article is a complement to \cite{SSW}. In that paper, the semistable
reduction of smooth plane quartics with non-hyperelliptic stable reduction is
reduced to two steps: first one computes a GIT-stable plane model, and then one
resolves the cusps of its special fiber. The first step is treated in \cite{SternWewers} and \cite{SSW};
the present paper gives the explicit local construction needed for the second
step.

More precisely, let \(\X_0\) be a plane curve model over a complete discretely
valued field \(K\), and suppose that a point \(P\) on the special fiber is a cusp.
We construct, after a finite separable extension of \(K\), a weighted blow-up
which is the stable resolution of \(\X_0\) at \(P\). The exceptional component is
an explicit semistable Weierstrass cubic and is the \(1\)-tail which appears in
the stable reduction of the global quartic.

Although the main application in this paper is to the cusps appearing in
\cite{SSW}, we formulate the problem for a general smoothing of a plane curve
singularity. The point is that the stable resolution can be constructed locally:
the exceptional component and the required field extension are determined
directly from the local equation of the smoothing, without first constructing a
global semistable model. We expect this local approach to extend to other plane
curve singularities.

\subsection{Stable resolution of a plane curve singularity}
\label{subsec:stable_resolution}

Let \(K\) be complete with respect to a discrete valuation \(v_K\), with valuation
ring \(\OO_K\) and algebraically closed residue field \(k\). Let
\[
    \X_0=V(f)\subset \AA^2_{\OO_K}
\]
be an integral affine plane curve model, with \(f\in\OO_K[x,y]\) primitive. Let
\(P\in \X_{0,s}\) be a closed point on the special fiber. We assume that \(\X_0\) is normal in a
neighborhood of \(P\), and smooth over \(\OO_K\) away from \(P\) in that
neighborhood. The above assumptions are preserved after finite extension of \(K\). 

A stable resolution of \(\X_0\) at \(P\) is a modification
\[
\X\longrightarrow \X_0
\]
which is an isomorphism away from \(P\), whose special fiber is obtained by
replacing \(P\) by an exceptional divisor \(E\), such that the strict transform of
\(\X_{0,s}\) meets \(E\) transversely in smooth points
\(\tilde P_1,\ldots,\tilde P_r\), and such that
\[
(E,\{\tilde P_1,\ldots,\tilde P_r\})
\]
is stably marked, in the sense of \cite[Definition 1.1]{SSW}. We call this marked curve the {\em tail} of the stable resolution. See also \cite[Definition 3.1]{hasset2000}.

The stable resolution, if it exists, is unique up to unique isomorphism. Moreover,
after a finite separable extension of \(K\), it exists. This follows from the
stable modification theorem of Temkin (\cite{Temkin2010}), or from the classical semistable reduction
theorem applied locally, e.g.\ as in \cite[\S 2]{liu2006stable}; see also \cite[\S 3.3]{SSW}. If \(L/K\) is chosen Galois and the stable resolution is defined over \(L\), uniqueness gives a monodromy action of \(\operatorname{Gal}(L/K)\) on the tail. The fixed field of the kernel is the minimal extension over which the stable resolution is defined; in particular, for the minimal extension the monodromy action is faithful.

\medskip
The general local problem is to make this stable resolution explicit: construct a
finite extension of \(K\), describe the tail, and give enough coordinate data to
determine the monodromy action. The present paper solves this
problem in the simplest nontrivial case, namely if the singularity is a cusp.

\subsection{The cusp case}\label{subsec:cusp_case}

We now assume that \(P\in \X_{0,s}\) is a cusp, i.e.
\[
\widehat{\OO}_{\X_{0,s},P}\cong k[[x,y]]/(y^2-x^3).
\]
After choosing suitable affine coordinates centered at \(P\), the special fiber
has the form
\begin{equation} \label{eq:cusp_normal_form}
\bar f(x,y)=y^2-x^3+\text{terms of \((2,3)\)-weighted degree \(>6\)}.
\end{equation}
We call such coordinates admissible. For admissible coordinates set
\[
t_{\max}(x,y)
:=
\min_{2i+3j<6}
\frac{v_K(a_{i,j})}{6-2i-3j},
\qquad \text{where}\;\;
f=\sum a_{i,j}x^iy^j.
\]
After a finite extension, choose \(\Pi\in K\) with
\(v_K(\Pi)=t_{\max}(x,y)\), and consider the weighted blow-up of
\((\Pi,x,y)\) with weights \(1,2,3\).

Our first main result is the following theorem. The proof we give is completely independent of the usual Stable Reduction Theorem of Deligne and Mumford.

\begin{theorem}
	\label{thm:cusp_resolution}
	Let \((\mathcal X_0,P)\) be as above. After
	replacing \(K\) by a finite separable extension, there exists an admissible
	coordinate system \((x,y)\) such that the weighted blow-up of
	\((\Pi,x,y)\), with \(v_K(\Pi)=t_{\max}(x,y)\), is a stable resolution of
	\(\mathcal X_0\) at \(P\).
	
	More precisely, the exceptional divisor is an irreducible semistable curve
	\(\bar X_1\) of arithmetic genus one, and it meets the strict transform
	\(\bar X_0\) in one ordinary node. The thickness of this node is
	\(t_{\max}(x,y)\).
\end{theorem}

The exceptional component is explicit. With
\[
\bar a_{i,j}:=\overline{\Pi^{2i+3j-6}a_{i,j}},
\]
it is the Weierstrass cubic
\[
y^2+\bar a_{1,1}xy+\bar a_{0,1}y
=
x^3-\bar a_{2,0}x^2-\bar a_{1,0}x-\bar a_{0,0}.
\]
Thus the local problem is reduced to choosing the coordinates so that this cubic
is semistable.

The proof has two parts. First we show that \(t_{\max}\) attains a maximum,
after allowing finite extensions of \(K\). Coordinates achieving this maximum
are called optimal, and we show that for optimal coordinates the exceptional cubic is
semistable.

For computations we use a more rigid criterion. Instead of searching directly for optimal coordinates, we impose three explicit coefficient conditions on the exceptional Weierstrass cubic. The solutions form an auxiliary scheme of coordinate changes, which we call the rigidifying scheme; on the open locus where the resulting cubic is in so-called strong rigidified form, this scheme is finite étale. Thus the search for suitable coordinates becomes a finite separable algebraic problem, which is the basis of the implementation.

The construction is implemented in SageMath~\cite{SageMath}.
The relevant code is part of the package \href{https://github.com/kst3rn/StabilityFunction}{\texttt{StabilityFunction}} (\cite{KletusGitHub}). The routine \texttt{resolve\_cusp} computes a field
extension over which the local resolution exists, the value \(t_{\max}\), the
exceptional Weierstrass cubic, and the corresponding coordinate-change data.
In the global quartic algorithm of \cite{SSW}, this local routine is applied
independently at the cusps of a GIT-stable special fiber and supplies the missing
local step from the GIT-stable plane model to the stable model.

\subsection{Relation with Hassett's work, and future direction}

The construction in this paper also suggests a broader direction. In
\cite{SternWewers} and \cite{SSW}, the search for a good model is organized by
varying the projective embedding of the curve; globally this variation is naturally
parametrized by the Bruhat--Tits building. In the present local situation the same
idea appears in a more elementary form: one varies the admissible coordinates
\((x,y)\) until the associated weighted blow-up has semistable exceptional
component. This point of view is close in spirit to Hassett's study of local stable
reduction of plane curve singularities, where stable reduction is treated through
embedded modifications of the ambient surface, and where toric and weighted
projective geometry naturally appear for toric and quasitoric singularities
\cite{hasset2000}. Although the goals are different, Hassett's results suggest that
the method developed here for cusps should extend to more general plane curve
singularities, with the single weighted blow-up replaced by suitable toric
modifications of the ambient surface.

\subsection*{Structure of the paper}
Section~\ref{sec:weierstrass} recalls the facts on Weierstrass cubics needed for
the exceptional component, including rigidified Weierstrass normal form and the
finite étaleness of the strong rigidifying locus. In
Section~\ref{sec:blow-up} we construct the weighted blow-up associated with
admissible coordinates and compute its special fiber.
In
Section~\ref{sec:existence_resolution} we prove the existence of optimal
coordinates and hence the stable-resolution theorem.
In Section \ref{sec:rigidification} we replace the nonconstructive optimality argument by an effective
rigidification procedure. Section \ref{subsec:rigidified} proves the existence of rigidified
coordinates, while Section \ref{subsec:monodromy} explains how the monodromy action determines
the minimal field of definition of the stable resolution.
Finally, Section~\ref{sec:implementation} discusses the implementation and its
use in stable reduction of plane quartics.

\subsection*{Convention on coordinate changes.}

Projective coordinates on
\(\PP^2_R\) will be written as row vectors, in the order
\[
(z,x,y),
\qquad
\PP^2_R=\Proj R[z,x,y].
\]
Then a {\em change of coordinates} is represented by a matrix \(g\in \GL_3(R)\). Given projective coordinates \((z,x,y)\), the new
coordinates \((z',x',y')\) are defined by
\[
(z',x',y'):=(z,x,y)\cdot g^{-1}.
\]
We write this as
\[
(z',x',y')=\lexp{g}{(z,x,y)}.
\]
Equivalently,
\[
(z,x,y)=(z',x',y')\cdot g.
\]
Thus, if \(F(z,x,y)=0\) is the equation of a plane curve in the old coordinates,
then its equation in the new coordinates is
\[
\lexp{g}{F}:=F((z,x,y)\cdot g),
\]
after renaming \((z',x',y')\) as \((z,x,y)\). We indicate this substitution by
\[
(z,x,y)\leftarrow (z,x,y)\cdot g.
\]

For affine coordinates we use the chart \(z=1\). The affine coordinate changes
are obtained by restricting to matrices which preserve this chart, i.e. matrices
whose first column is \((1,0,0)^t\). 
On the affine chart \(z=1\), the substitution convention becomes
\[
(x,y)\leftarrow (x,y)\cdot g,
\]
where this means
\[
x\leftarrow \alpha+ax+by,\qquad
y\leftarrow \beta+cx+dy.
\]
Equivalently, the new affine coordinates \((x',y')\) are determined by
\[
(1,x',y')=(1,x,y)\cdot g^{-1}.
\]
With this convention, affine coordinate systems form a left principal homogeneous
space under the affine linear group.


\section{Weierstrass cubics} \label{sec:weierstrass}

This section collects some elementary facts on Weierstrass equations for plane cubics which will be used later. Throughout this section we assume that \(k\) is an algebraically closed field, and we write
\(p=\operatorname{char}(k)\).

\subsection{Weierstrass equations}

We first recall the standard Weierstrass presentation of an integral genus-one
curve with a chosen smooth point. The coordinate changes preserving this form
are encoded by the Weierstrass group.

\begin{definition}\label{def:WNF}
	A plane cubic $C\subset \PP^2_k$ 
	is said to be \emph{in Weierstrass normal form} (WNF) if it is given by a homogeneous equation of the form
	\begin{equation}\label{eq:WNF_homog}
		y^2 z + a_1 x y z + a_3 y z^2
		=
		x^3 + a_2 x^2 z + a_4 x z^2 + a_6 z^3,
	\end{equation}
	with $a_i\in k$.
\end{definition}

If $C\subset\PP^2_k$ is in WNF then $C$ is an integral projective curve of arithmetic genus one, and $O:=[0:1:0]$ is a smooth point of $C$. Moreover, $C$ is either smooth, or has a unique singularity, which is either a node or a cusp. If $C$ has a node (resp.\ a cusp), we say that $C$ is {\em nodal} (resp.\ \emph{cuspidal}). If it is smooth or nodal, we call it {\em semistable}.

We note that if $C$ is semistable, then $(C,O)$ is a {\em $1$-pointed stable curve of genus one}.

\begin{definition} \label{def:W-group}
  The {\em Weierstrass group} is the closed subgroup $W\subset\PGL_3(k)$ given by matrices of the form
  \[
       \begin{pmatrix}
       	   1 & \alpha & \beta \\
       	   0 & u^2 & u^2\gamma \\
       	   0 & 0   & u^3 
       \end{pmatrix}
  \]	
  with $\alpha,\beta,\gamma\in k$ and $u\in k^\times$.
  We let $U\subset W$ denote the unipotent subgroup defined by $u=1$.
\end{definition}

If $C\subset \PP^2_k$ is a cubic in WNF and $g\in W$, then the cubic $C'=\lexp{g}{C}$ obtained from $C$ by the coordinate change\footnote{Note that we use the convention mentioned at the beginning that we order the variables as $(z,x,y)$.}  
\[
    z \leftarrow z,\; x \leftarrow \alpha z + u^2x, \; y \leftarrow \beta z + u^2\gamma x + u^3 y, 
\]
corresponding to $g$ is again in WNF. Equivalently, $C'=g^*(C)$, where $g^*$ denotes the automorphism of $\PP^2_k$ induced by $g$. 

\begin{proposition} \label{prop:WNF}
  Let $C$ be an integral projective curve over $k$, of arithmetic genus one, and $O\in C$ a smooth point. 
  \begin{enumerate}[(i)]
  \item 
    There exists an embedding $\phi:C\inj\PP^2_k$ such that $\phi(C)$ is a cubic in Weierstrass normal form and $\phi(O)=[0:1:0]$.
  \item 
    If $\phi':C\inj\PP^2_k$ is another embedding as in (i), then there exists a unique element $g\in W$ such that $\phi'=g^*\circ \phi$.
  \end{enumerate}	
\end{proposition}	

\begin{proof}
If $C$ is smooth, this is \cite[Proposition III.3.1]{SilvermanAEC}. However, the proof goes through in the general case as well. Indeed, the proof only uses Riemann--Roch for the divisor \(3O\), which
applies to integral projective Gorenstein curves of arithmetic genus one.
\end{proof}	

\begin{corollary} \label{cor:wnf_stabilizer}
  If $C\subset\PP^2_k$ is a cubic in Weierstrass normal form then the stabilizer of $C$ in $W$ naturally identifies with the automorphism group of the pointed curve $(C,O)$. 
\end{corollary}

\subsection{Rigidified Weierstrass equations}

We now impose some additional coefficient conditions on a Weierstrass equation which make the normal form sufficiently rigid.

\begin{definition} \label{def:RWNF}
  Let $C\subset\PP^2_k$ be a cubic in Weierstrass normal form, with equation \eqref{eq:WNF_homog}. We say that $C$ is in \emph{weak rigidified Weierstrass normal form} (wRWNF) if either $p=2$ and
  \begin{equation} \label{eq:rwnf_p=2}
     a_2=a_4=a_6=0,
  \end{equation}
  or $p\neq 2$ and 
  \begin{equation} \label{eq:rwnf_p!=2}
  	a_1=a_3=a_6=0.
  \end{equation}
  We say that it is in \emph{strong rigidified Weierstrass normal form} (sRWNF) if in addition, we have $a_3\neq 0$ for $p=2$ resp.\ $a_4\neq 0$ for $p\neq 2$. 
  
  We will usually omit the qualifier {\em strong} when referring to a strong rigidified Weierstrass normal form. 
\end{definition}

\begin{remark}
  One easily checks the following: if $C\subset\PP^2_k$ is in weak rigidified Weierstrass normal form then $P=(0,0)$ is a point on $C$. The strongness condition is then equivalent to $P$ being a smooth point. 	
  
  Assume that $C$ is in (strong) rigidified Weierstrass normal form and that, moreover, $C$ is smooth. Then for $p=2$ we have
  \[
      C:\; y^2+a_1xy+a_3y = x^3, \quad a_3(a_3+a_1^3)\neq 0,
  \]
  and $P=(0,0)$ is a $3$-torsion point of the elliptic curve $(C,O)$. Up to normalizing $a_3=1$ by the transformation $x\leftarrow u^2x$, $y\leftarrow u^3y$ for some $u\in k^\times$ with $u^3=a_3$, $(C,O)$ is an elliptic curve in {\em Deuring normal form}.
  
  Similarly, for $p\neq 2$ we have
  \[
     C:\; y^2 = x^3 + a_2x^2+a_4x, \quad a_4(4a_4-a_2^2)\neq 0,
  \] 
  $P=(0,0)$ is a $2$-torsion point, and up to the normalization $a_4=1$, $(C,O)$ is an elliptic curve in {\em Montgomery normal form}.
\end{remark}

\begin{proposition} \label{prop:RWNF}
	Let \(C\subset\PP^2_k\) be a cubic in Weierstrass normal form.
	\begin{enumerate}[(i)]
		\item
		There exists an element \(g\in U\) such that \(C':=\lexp{g}{C}\) is in
		weak rigidified Weierstrass normal form. If \(C\) is semistable, then \(g\)
		may be chosen such that \(C'\) is in strong rigidified Weierstrass normal form.
		\item
		If \(C\) is in weak rigidified Weierstrass normal form, with equation
		\eqref{eq:WNF_homog}, then \(C\) is semistable if and only if at least one
		of the coefficients not forced to vanish in the definition of wRWNF is nonzero.
		Equivalently, \(C\) is cuspidal if and only if
		\[
		C:\; y^2z=x^3.
		\]
	\end{enumerate}
\end{proposition}

\begin{proof}
	We work on the affine chart \(z=1\).
	
	Assume first that \(p\neq 2\). Completing the square gives
	\[
	y^2=f(x), \qquad f=x^3+a_2x^2+a_4x+a_6 .
	\]
	Choose a root \(r\) of \(f\) and translate \(x\leftarrow x+r\). Then the
	constant term vanishes, so the equation is in weak RWNF. If \(C\) is semistable,
	then \(f\) is not the cube of a linear polynomial, hence has a simple root. Choosing
	\(r\) to be such a root gives \(a_4'=f'(r)\neq0\), so the form is strong.
	
	Now assume \(p=2\). We claim that there is a point \(P\neq O\) and a line
	\(\ell\not\ni O\) such that \(\ell\cdot C=3P\), with \(P\) smooth if \(C\) is
	semistable. Indeed, in the semistable case choose a nonzero point
	\(P\in C^{\rm sm}[3]\); in the cuspidal case take \(P\) to be the cusp and
	\(\ell\) its tangent line. After a unipotent coordinate change
	\[
	x\leftarrow x+r,\qquad y\leftarrow y+sx+t,
	\]
	we may assume \(P=(0,0)\) and \(\ell\) is the line \(y=0\). Intersecting the new
	Weierstrass equation with \(y=0\) gives
	\[
	x^3+a_2'x^2+a_4'x+a_6'=x^3,
	\]
	hence \(a_2'=a_4'=a_6'=0\). Thus the equation is in weak RWNF, and it is
	strong precisely when \(P\) is smooth.
	
	It remains to prove (ii). For \(p\neq2\), a weak RWNF equation has the form
	\[
	y^2=f(x)=x(x^2+a_2x+a_4).
	\]
	A singularity is cuspidal exactly when \(f\) has a triple root. Since \(0\) is a
	root, this happens exactly when \(f=x^3\), i.e. \(a_2=a_4=0\).
	
	For \(p=2\), write
	\[
	F=y^2+(a_1x+a_3)y+x^3 .
	\]
	If \(a_1=a_3=0\), this is \(y^2=x^3\). Otherwise, if \(a_1=0\), then
	\(\partial F/\partial y=a_3\neq0\), so \(C\) is smooth. If \(a_1\neq0\), scale
	so that \(a_1=1\). The only possible singular point is
	\((a_3,a_3^2)\), and it lies on \(C\) only for \(a_3=0\) or \(a_3=1\). Expanding
	at that point gives tangent cone \(y(y+x)\) in the first case and
	\(y^2+xy+x^2\) in the second; both split into two distinct linear factors over
	\(k\). Hence the singularity is a node. This proves (ii).
\end{proof}

\subsection{The rigidifying scheme}

Fix a Weierstrass cubic \(C\subset\PP^2_k\). We study the scheme of unipotent coordinate
changes which put \(C\) into weak rigidified form. We show that the strong part of this
scheme is the transverse, hence étale, part of the rigidifying equations.

In the following, we will interpret a cubic $C\subset\PP^2_k$ in Weierstrass normal form with equation \eqref{eq:WNF_homog} as a point $(a_1,a_2,a_3,a_4,a_6)\in \AA_k^5$. Let us fix one $C\in\AA^5_k$.
We identify the unipotent subgroup $U\subset W$ (Definition \ref{def:W-group}) with the affine space $\AA^3_k$ with coordinates $\alpha,\beta,\gamma$. Define
\[
    S := \{ g\in U \mid \text{$\lexp{g}{C}$ is in weak RWNF}\}\subset U=\AA^3_k.
\]
Explicitly, $S\subset\AA^3_k$ is defined by three equations
\[
    A_i(\alpha,\beta,\gamma)=0, \quad i\in I,
\]
where $I=\{2,4,6\}$ if $p=2$ and $I=\{1,3,6\}$ if $p\neq 2$, and where $A_i\in k[\alpha,\beta,\gamma]$ are the coefficients of 
\[
    F(x+\alpha z, y + \gamma x + \beta z, z).
\]
We consider $S$ as a closed, but not necessarily reduced subscheme of $\AA^3_k$, defined by these equations. 

Let $S'\subset S$ be the open subscheme corresponding to {\em strong} RWNFs. It is defined by the condition $A_3\neq 0$ for $p=2$ and $A_4\neq 0$ for $p\neq 2$. 

\begin{proposition} \label{prop:srwnf_is_etale}
  If $C$ is semistable, then the scheme $S'$ is nonempty and finite \'etale over $k$. In particular, $S'$ is reduced.
\end{proposition}	

\begin{proof}
	That $S'$ is nonempty follows from Proposition \ref{prop:RWNF} (i). For \'etaleness, it is enough to show that, at every point of
	\(S'\), the differentials of the three defining equations form a basis
	of the cotangent space of \(U\).
	
	Let \(g_0\in S'(k)\). Replacing \(C\) by \(\lexp{{g_0}}{C}\), and
	translating \(U\) by \(g_0\), we may assume that \(g_0=1\) and that \(C\) itself
	is already in strong RWNF.
	
	Suppose first that \(p\neq 2\). Then
	\[
	C:\quad y^2=x^3+a_2x^2+a_4x,
	\qquad a_4\neq 0.
	\]
	Consider the infinitesimal coordinate change
	\[
	x\leftarrow x+\varepsilon\alpha,\qquad
	y\leftarrow y+\varepsilon(\gamma x+\beta),
	\qquad \varepsilon^2=0.
	\]
	Modulo \(\varepsilon^2\), the coefficients defining weak RWNF change by
	\[
	A_1=2\gamma\varepsilon,\qquad
	A_3=2\beta\varepsilon,\qquad
	A_6=a_4\alpha\varepsilon.
	\]
	Thus
	\[
	\det
	\frac{\partial(A_1,A_3,A_6)}
	{\partial(\gamma,\beta,\alpha)}
	=
	4a_4\neq 0.
	\]
	Hence the equations \(A_1=A_3=A_6=0\) cut out an étale subscheme at the
	identity.
	
	Now suppose that \(p=2\). Then
	\[
	C:\quad y^2+(a_1x+a_3)y=x^3,
	\qquad a_3\neq 0.
	\]
	For the same infinitesimal coordinate change, the coefficient transformation
	formulae give
	\[
	\begin{aligned}
		A_2 &= (\alpha+a_1\gamma)\varepsilon,\\
		A_4 &= (a_3\gamma+a_1\beta)\varepsilon,\\
		A_6 &= a_3\beta\varepsilon .
	\end{aligned}
	\]
	Therefore
	\[
	\det
	\frac{\partial(A_2,A_4,A_6)}
	{\partial(\alpha,\gamma,\beta)}
	=
	a_3^2\neq 0.
	\]
	Hence the equations \(A_2=A_4=A_6=0\) cut out an étale subscheme at the
	identity.
	
	We have shown that \(S'\) is étale over \(k\) at all of its points.
	In particular, it is zero-dimensional. Since \(S'\) is of finite type
	over \(k\), it is finite over \(k\). Hence \(S'\) is finite étale.
\end{proof}


\section{A local weighted blow-up construction}
\label{sec:blow-up}

Let \((\X_0,P)\) be a smoothing of a cusp, as in the introduction. The goal of
this section is to attach to any suitable choice of local coordinates $(x,y)$ at \(P\) and a further parameter $t$  a
weighted blow-up
\[
\X = \X((x,y),t)\longrightarrow \X_0 .
\]
We compute its special fiber explicitly. The exceptional component will be a
Weierstrass cubic, and the only remaining question will be whether this cubic is
semistable. This reduces the construction of the stable resolution to the problem
of choosing the coordinates \((x,y)\) and the parameter \(t\) appropriately.

\subsection{Admissible coordinates} \label{subsec:admissible}

We recall the assumptions made in \S \ref{subsec:stable_resolution} and \S \ref{subsec:cusp_case}. So $K$ is a field complete with respect to a discrete valuation $v_K$, and $\X_0=V(f)\subset\AA_{\OO_K}^2$ is an affine plane curve model defined by a polynomial equation
\begin{equation} \label{eq:f}
   f = \sum_{i,j} a_{i,j} x^iy^j = 0,
\end{equation}
with $a_{i,j}\in\OO_K$. Let $X:=\X_{0,K}\subset \AA^2_K$ denote its generic fiber. We assume that $f$ is primitive with respect to $v_K$. So the special fiber $\X_{0,s}=V(\bar{f})\subset \AA_k^2$ is defined by the polynomial equation 
\begin{equation} \label{eq:fb}
  \bar{f} = \sum_{i,j} \bar{a}_{i,j} x^iy^j = 0,
\end{equation}
where $\bar{a}_{i,j}\in k$ is the image of $a_{i,j}$ in the residue field $k$.

We fix a closed point $P\in\X_{0,s}$, and assume that there exists an open neighborhood $\U_0\subset \X_0$ which is normal and such that $\U_0\backslash\{P\}$ is smooth over $\OO_K$. We also assume that $P\in\X_{0,s}$ is a cusp (\cite[Definition 2.1]{SSW}).

The coefficients \(a_{i,j}\) in \eqref{eq:f} depend on the chosen affine
coordinates \((x,y)\). The proof of our main theorem rests on choosing these
coordinates carefully. We first put the cusp into a fixed normal form and study
the weighted blow-up attached to such a choice; the next section explains how
to vary the coordinates so that this blow-up gives the stable resolution.

\begin{lemma} \label{lem:admissible}
\begin{enumerate}
\item
  There exist affine coordinates $x,y$ on $\AA_{\OO_K}^2$, centered at $P$, such that, after multiplying the equation $f$ by a unit in $\OO_K^\times$,
  \begin{equation} \label{eq:admissible_condition}
     v_K(a_{i,j})\;\; \begin{cases}
     	\;\;= 0, & 2i+3j=6, \\
     	\;\;> 0, & 2i+3j<6, \\
     	\;\;\geq 0, & 2i+3j>6.
     \end{cases} 
  \end{equation} 
\item 
  Whenever Condition \eqref{eq:admissible_condition} holds, there is at least one coefficient $a_{i,j}$ with $2i+3j<6$ that does not vanish. 
\end{enumerate}	
\end{lemma}	

\begin{proof}
  By \cite[Lemma 2.2]{SSW} and the assumption that $P$ is a cusp, there exists a coordinate system $(x,y)$ for $\AA_k^2$ such that $P=(0,0)$ and $\X_{0,s}=V(\bar{f})$ is defined by a polynomial $\bar{f}\in k[x,y]$ of the form 
  \begin{equation} 
  	\bar{f} = y^2 - x^3 +\text{terms of \((2,3)\)-weighted degree \(>6\)}.
  \end{equation} 
  In particular, \eqref{eq:admissible_condition} holds. Lifting $(x,y)$ to a coordinate system for $\AA_{\OO_K}^2$ proves (i).
  
  Assume that $P=(0,0)\in\AA_k^2$ and that $a_{i,j}=0$ for all $i,j$ with $2i+3j<6$. Then the point $(0,0)\in \AA_K^2$ is a singular point of $X=V(f)$. The Zariski closure of this point in $\X_0$ contains $P$ and hence is contained in any open neighborhood of $P$. This contradicts our assumption made on $\X_0$, and proves (ii).
\end{proof}
	
\begin{definition} \label{def:admissible}
  An affine-linear coordinate system $(x,y)$ for $\AA_{\OO_K}^2$ such that $P=(0,0)\in\AA_k^2$ and \eqref{eq:admissible_condition} holds is called {\em admissible}. It is called {\em strictly admissible} if, moreover, the equation $f$ can be normalized such that
  \begin{equation} \label{eq:strictly_admissible_condition}
      a_{3,0} = -1, \quad a_{0,2}=1. 
  \end{equation}
\end{definition}

\begin{remark} \label{rem:scaling}
  Let $(x,y)$ be admissible coordinates. Then $a_{3,0},a_{0,2}\in\OO_K^\times$ are units, and after the coordinate change
  \[
       x\leftarrow ux,\quad y\leftarrow uy,
  \]
  with $u:=-a_{0,2}/a_{3,0}\in\OO_K^\times$, and multiplying $f$ with $1/a_{0,2}u^2$, \eqref{eq:strictly_admissible_condition} holds. Clearly, the valuations $v_K(a_{i,j})$ and therefore the validity of Condition \eqref{eq:admissible_condition} are preserved by this operation. Therefore, admissible coordinates become strictly admissible by a suitable scaling with units. 
\end{remark}

Depending on the choice of admissible coordinates $(x,y)$ and a further parameter $t\in\QQ_{>0}$, we construct a certain modification
\[
\X=\X((x,y),t)\to\X_0
\]
of $\X_0$ which is a candidate for the stable resolution of the cusp $P$; it satisfies all requirements of a stable resolution, with the possible exception of the condition that the exceptional fiber be semistable. In fact, if we denote the exceptional divisor of $\X\to\X_0$ by $\Xb_1$ and $\tilde{P}\in\Xb_1$ is the point where $\Xb_1$ meets the strict transform $\Xb_0$ of $\X_{0,s}$, then the pointed curve $(\Xb_1,\tilde{P})$ occurs naturally as a cubic in Weierstrass normal form (Definition \ref{def:WNF}). Therefore,  $\X\to\X_0$ is the stable resolution of the cusp $P$ if and only if $\Xb_1$ is semistable, i.e.\ not cuspidal. 

The construction of the modification $\X\to\X_0$ is done via a {\em weighted blow-up}, see e.g.\ \cite{abramovich2023birational}. As we could not find an accessible, down-to-earth reference, we devote the rest of this section to the details of the construction.

\subsection{The ambient weighted blow-up}
\label{subsec:blow-up}

We first construct the ambient modification of \(\AA^2_{\OO_K}\) in which the
candidate resolution will live. 

Fix a strictly admissible coordinate system $(x,y)$. Then
\[
   \X_0=\Spec(A/(f)),
\]
where $A:=\OO_K[x,y]$ and
\[
   f=\sum_{i,j} a_{i,j}x^iy^j
\]
such that \eqref{eq:admissible_condition} holds. The point $P\in\X_0\subset\AA_{\OO_K}^2$ corresponds to the maximal ideal $\m:=(\m_K,x,y)\lhd A$, which contains $f$. 

Set
\begin{equation} \label{eq:t_def}
	t_{\max}(x,y) := \min_{2i+3j < 6} \frac{v_K(a_{i,j})}{6-2i-3j}.
\end{equation}
This is a well-defined rational number by Lemma \ref{lem:admissible} (ii). Using Condition \ref{eq:admissible_condition} we immediately check the following claims.

\begin{lemma} \label{lem:t_max}
\begin{enumerate}
\item 
  $t_{\max}>0$.
\item 
  Choose a rational number  $t$ such that $0<t\leq t_{\max}(x,y)$. 
  Then
  \begin{equation} \label{eq:weight_cond1}
  	\min_{i,j} \left(\frac{v_K(a_{i,j})}{t} + 2i + 3j\right) = 6,
  \end{equation}
  and this minimum is attained for $(i,j)=(3,0),(0,2)$.
\item 
  We have $t=t_{\max}(x,y)$ if and only if the minimum in \eqref{eq:weight_cond1} is also attained for a pair $(i,j)$ with $2i+3j<6$.
\end{enumerate}	
\end{lemma}

We fix a rational number $t$ as in the lemma. After replacing $K$ by some finite extension, we may assume that there exists an element $\Pi\in K$ with $v_K(\Pi)=t$.
Let
\[
   q:\P\to\AA_{\OO_K}^2
\]
be the \emph{weighted blow-up} along the regular sequence $\Pi,x,y$, with weights
\[
w(\Pi)=1,\quad w(x)=2,\quad w(y)=3.
\]
Explicitly, $\P=\Proj(R)$, where $R$ is the graded ring 
\begin{equation} \label{eq:R_def}
R := \bigoplus_{\ell\geq 0} I_\ell T^\ell \subset A[T],
\end{equation}
with
\[
I_\ell := ( \Pi^kx^iy^j \mid k+2i+3j\geq \ell) \lhd A.
\]

We set $\tilde{\Pi}:=\Pi T$, $\tilde{x}:=xT^2$, and $\tilde{y}:=yT^3\in R$. 

\begin{lemma} \label{lem:blow-up}
\begin{enumerate}
\item
  The morphism $q:\P\to\AA_{\OO_K}^2$ is projective and an isomorphism away from $P$. 
\item 
  The three open affine subsets $D_+(\tilde{\Pi})$, $D_+(\tilde{x})$ and $D_+(\tilde{y})$ cover $\P$.
\item 
  We have $D_+(\tilde{\Pi})=\Spec A_\Pi,$ where
  \[
     A_\Pi = \OO_K[x_1,y_1]
  \] 
  is a polynomial ring with generators 
  \[
     x_1:=\frac{\tilde{x}}{\tilde{\Pi}^2} = \frac{x}{\Pi^2}, \quad y_1:=\frac{\tilde{y}}{\tilde{\Pi}^3} = \frac{y}{\Pi^3}.
  \]
\item 
  The intersection $D_+(\tilde{x})\cap D_+(\tilde{y})=\Spec A_{xy}$ is affine, and the $\OO_K$-algebra $A_{xy}$ has the presentation
  \[
     A_{xy} = \OO_K[u^{\pm 1}, r,s]/(rs-\Pi),
  \]
  where 
  \[
  u:=\frac{\tilde y^2}{\tilde x^3}=\frac{y^2}{x^3},
  \qquad
  r:=\frac{yT^2}{xT^2}=\frac yx,
  \qquad
  s:=\frac{\tilde\Pi\tilde x}{\tilde y}=\frac{\Pi x}{y}.
  \]
\end{enumerate}	
  In (iii) and (iv) we identify rational functions on $\P$ (defined as quotients of homogeneous elements of $R$ of the same degree) with the corresponding elements of $K(x,y)$, the function field of $\AA_{\OO_K}^2$.
\end{lemma}

\begin{proof}
It is an elementary exercise to check that $R$ is generated as $A$-algebra by the following six homogeneous elements:
\begin{equation} \label{eq:R_generators}
 \Pi T, xT, yT, xT^2, yT^2, yT^3.
\end{equation}
In particular, $R$ is a finitely generated graded $A$-algebra. Hence there exists $m\geq 1$ such that the graded subring $R'=\oplus_d R_{md}$ is generated by elements of degree $m$. As $\P=\Proj(R)=\Proj(R')$, \cite[Corollary II.5.16 (b)]{HartshorneAG} shows that $q$ is projective. 

For $a\in\{\Pi,x,y\}$, set $D(a)=\Spec A_a$, with $A_a:=A[a^{-1}]$. Then
\[
q^{-1}(D(a)) = \Proj R_a, \quad R_a= \oplus_{\ell\geq 0} I_\ell A_a.
\]
But any sufficiently high power of $a$ has weight $\geq \ell$ and is invertible in $A_a$, and so $I_\ell A_a=A_a$. This shows that $q^{-1}(D(a))\cong D(a)$. We conclude that $q$ is an isomorphism over
\[
\AA_{\OO_K}^2\backslash\{P\} = D(\Pi)\cup D(x)\cup D(y).
\]

To prove (ii), assume $\p\in\Proj(R)$ contains all three elements $\tilde{\Pi},\tilde{x},\tilde{y}$. Then $\p$ contains all six generators from \eqref{eq:R_generators}, because
\[
   (xT)^2 = x\tilde{x},\; (yT)^3 = y^2\tilde{y},\; (yT^2)^3 = y\tilde{y}^2
\]
all lie in $\p$. It follows that $R_+\subset\p$, contradiction. 

For (iii), we compute
\[
A_\Pi=(R_{\tilde\Pi})_0\subset K(x,y).
\]
Since \(\deg(\tilde\Pi)=1\), every homogeneous element of \(A_\Pi\) is a sum of
monomials of the form
\[
\frac{\Pi^kx^iy^jT^n}{(\Pi T)^n}
=
\frac{\Pi^kx^iy^j}{\Pi^n},
\qquad
k+2i+3j\ge n.
\]
With
\[
x_1:=\frac{x}{\Pi^2},
\qquad
y_1:=\frac{y}{\Pi^3},
\]
this monomial becomes
\[
\Pi^{k+2i+3j-n}x_1^iy_1^j,
\]
and the exponent \(k+2i+3j-n\) is nonnegative. Hence
\[
A_\Pi\subset \OO_K[x_1,y_1].
\]
The reverse inclusion follows from
\[
x_1=\frac{\tilde x}{\tilde\Pi^2},
\qquad
y_1=\frac{\tilde y}{\tilde\Pi^3}.
\]
Thus
\[
D_+(\tilde\Pi)=\Spec \OO_K[x_1,y_1].
\]

For (iv), set
\[
   A_{xy}:=(R_{\tilde x\tilde y})_0\subset K(x,y).
\]
Then \(u,u^{-1},r,s\in A_{xy}\), and these elements satisfy
\[
   rs=\Pi.
\]
This gives an \(\OO_K\)-algebra homomorphism
\[
\OO_K[u^{\pm1},r,s]/(rs-\Pi)\longrightarrow A_{xy}.
\]
We show that it is an isomorphism. Injectivity is clear if we view both sides as subrings of $K(x,y)$. A monomial in \(A_{xy}\) is of the form
\[
\Pi^a x^b y^c,
\qquad
a\ge0,\quad b,c\in\ZZ,\quad a+2b+3c\ge0.
\]
Indeed, this follows directly from localizing at
\(\tilde x\tilde y=xyT^5\). Put
\[
    n:=a+2b+3c\geq 0, \qquad m:=-b-c.
\]
Then
\[
     \Pi^a x^b y^c = u^m r^n s^a.
\]
Hence \(u^{\pm1},r,s\) generate \(A_{xy}\), so the homomorphism above is an isomorphism.
\end{proof}

\begin{remark} \label{rem:weighted-blowup-references}
 The phrase ``weighted blow-up'' is used with slightly different conventions in the literature. In this paper it refers to the Proj of the weighted-order Rees algebra \eqref{eq:R_def}. This is the normal toric version of the weighted blow-up with weights \(1,2,3\). Since the available references either treat a more general framework or impose extra hypotheses not needed here, we have included the elementary affine-chart computations. Compare, for instance, \cite[Definition~2.2 and Proposition~2.4]{AndreattaLiftingWBU}.
\end{remark}

\subsection{The candidate resolution}
\label{subsec:candidate_resolution}

We define \(\X\subset\P\) as the schematic closure of \(q^{-1}(X)\). Since \(\X_0\) is the schematic closure of \(X\) in \(\AA^2_{\OO_K}\), the restriction of \(q\) induces a morphism
\[
   \X\to\X_0.
\]
This is the weighted blow-up of \(\X_0\) attached to the chosen strictly admissible coordinates \((x,y)\) and the parameter \(t\). It is our candidate for the stable resolution of the cusp \(P\).

\begin{proposition} \label{prop:cusp_resolution}
	\begin{enumerate}
	\item 
	  \(\X\to\X_0\) is projective and is an isomorphism away from \(P\). Moreover,
	  \(\X\) is flat over \(\OO_K\), with generic fiber \(X\).
	\item
	  Let $\Xb_0\subset\X_s$ denote the strict transform of $\X_{0,s}$. The natural map $\Xb_0\to\X_{0,s}$ is the desingularization of $\X_{0,s}$ at the cusp $P$.
	\item
	  Let $\Xb_1\subset\X_s$ be the exceptional fiber of $q$. Then $\Xb_1$ is an integral curve over $k$ of arithmetic genus one.
	\item
	  The curve $\Xb_1$ intersects $\Xb_0$ precisely in $\tilde{P}$, the inverse image of $P$. This is a smooth point of $\Xb_1$.
	\item 
	  The point $\tilde{P}$ is a node of $\X_s$, and the
	  parameter $t$ is equal to its thickness\footnote{If $\tilde{P}$ is a node of $\X_s$, then the complete local ring of $\X$ at $\tilde{P}$ has the form  $\widehat{\OO}_{\X,\tilde{P}}\cong\OO_K[[u,v]]/(uv-a)$, with $a\in\m_K$; the {\em thickness} of the node is defined as $v_K(a)$.}  in $\X$.
	\end{enumerate}
\end{proposition}

The proof of (iii) will reveal an explicit equation for the component $\Xb_1$ as a cubic in Weierstrass normal form, see \eqref{eq:exceptional_curve_equation}. 

We note the following immediate consequence of the proposition.

\begin{corollary} \label{cor:cusp_resolution}
	If $\Xb_1$ is semistable then $\X\to\X_0$ is the stable resolution of the cusp $P$.
\end{corollary}

Indeed, \(\Xb_1\) is irreducible, has arithmetic genus one and \(\tilde P\) is a smooth point
on it; hence \((\Xb_1,\tilde P)\) is a stable one-pointed curve of genus one if $\Xb_1$ is semistable.

In order to check whether $\Xb_1$ is semistable, there are three cases to consider: $\Xb_1$ may be smooth, or have exactly one node, or exactly one cusp. In the first two cases, $\Xb_1$ is semistable, and then $\X\to\X_0$ is the resolution of the cusp $P\in\X_{0,s}$. In Section \ref{sec:existence_resolution} we will show that, after a finite extension of $K$ and a careful choice of the admissible coordinate system $(x,y)$, we can force this to be true.

\begin{proof}[Proof of Proposition \ref{prop:cusp_resolution}]
	\emph{(i)}
	Since \(\X\) is the schematic closure of the generic fiber in \(\P\), it is \(\OO_K\)-torsion free, and so \(\X\) is
	flat over \(\OO_K\). The morphism \(\X\to\X_0\) is projective and an isomorphism
	away from \(P\), because the same is true for \(q:\P\to\AA^2_{\OO_K}\) by Lemma \ref{lem:blow-up} (i).
	
	\medskip
	We now use the affine charts from Lemma \ref{lem:blow-up} (iii) and (iv). Put
	\[
	w_{i,j}:=2i+3j,
	\qquad
	b_{i,j}:=\frac{a_{i,j}}{\Pi^{6-w_{i,j}}}\in\OO_K
	\quad\text{for }w_{i,j}<6.
	\]
	The integrality of \(b_{i,j}\) follows from the assumption \(t\leq t_{\max}(x,y)\) and the definition of $t_{\max}(x,y)$, see \eqref{eq:t_def}.
	
	On \(D_+(\tilde\Pi)=\Spec\OO_K[x_1,y_1]\) we have
	\[
	f(\Pi^2x_1,\Pi^3y_1)=\Pi^6F(x_1,y_1),
	\]
	where
	\[
	F
	=
	\sum_{w_{i,j}<6} b_{i,j}x_1^iy_1^j
	+
	\sum_{w_{i,j}\ge6} a_{i,j}\Pi^{w_{i,j}-6}x_1^iy_1^j
	\in\OO_K[x_1,y_1].
	\]
	As $v_K(a_{i,j})=0$ for $(i,j)=(3,0),(0,2)$ by \eqref{eq:admissible_condition}, $F$ is primitive. 
	Thus \(\X\cap D_+(\tilde\Pi)\) is given by \(F=0\). Reducing modulo
	\(\m_K\), we obtain the affine Weierstrass equation
	\begin{equation} \label{eq:exceptional_curve_equation}
		y_1^2+\bar a_1x_1y_1+\bar a_3y_1
		=
		x_1^3+\bar a_2x_1^2+\bar a_4x_1+\bar a_6,
	\end{equation}
	with
	\[
	\bar a_1:=\bar b_{1,1},\quad
	\bar a_2:=-\bar b_{2,0},\quad
	\bar a_3:=\bar b_{0,1},\quad
	\bar a_4:=-\bar b_{1,0},\quad
	\bar a_6:=-\bar b_{0,0}.
	\]
	Let \(\bar X_1\) be the closure of this affine curve in \(\X_s\).
	
	Next consider the overlap
	\[
	D_+(\tilde x)\cap D_+(\tilde y)
	=
	\Spec \OO_K[u^{\pm1},r,s]/(rs-\Pi),
	\]
	from Lemma \ref{lem:blow-up} (iv).
	On this chart we have
	\[
	x=u^{-1}r^2,\qquad y=u^{-1}r^3,\qquad \Pi=rs.
	\]
	Substituting into \(f\), one obtains
	\[
	f=r^6G(u,r,s),
	\]
	where
	\[
	G
	=
	\sum_{w_{i,j}<6}
	b_{i,j}u^{-i-j}s^{6-w_{i,j}}
	+
	\sum_{w_{i,j}\ge6}
	a_{i,j}u^{-i-j}r^{w_{i,j}-6}.
	\]
	Hence \(\X\) is given on this chart by \(G=0\). Modulo \(\m_K\), the two
	weight-six terms \(y^2\) and \(-x^3\) give
	\[
	\bar G
	=
	u^{-2}-u^{-3}
	+rA+sB
	=
	u^{-3}(u-1)+rA+sB
	\]
	for suitable \(A,B\in k[u^{\pm1},r,s]/(rs)\). In particular,
	\[
	\frac{\partial \bar G}{\partial u}(1,0,0)=1.
	\]
	Let \(\tilde P\) be the point \(u=1,r=s=0\). By the formal implicit function
	theorem, \(u\) can be eliminated in the completed local ring at \(\tilde P\), and
	we get
	\[
	\widehat{\OO}_{\X,\tilde P}
	\cong
	\OO_K[[r,s]]/(rs-\Pi).
	\]
	After reducing modulo \(\m_K\),
	\[
	\widehat{\OO}_{\X_s,\tilde P}
	\cong
	k[[r,s]]/(rs).
	\]
	Thus the two local branches of \(\X_s\) at \(\tilde P\) are \(s=0\) and \(r=0\).
	
	\emph{(ii)}
	The branch \(s=0\) is the strict transform \(\bar X_0\). Indeed, away from
	\(\tilde P\) it maps isomorphically to \(\X_{0,s}\setminus\{P\}\), and near
	\(\tilde P\) the map to \(\X_{0,s}\) is given by
	\[
	x=u^{-1}r^2,\qquad y=u^{-1}r^3,
	\]
	with \(u(0)=1\). Hence \(r=y/x\) is a local parameter on the normalization of the
	cusp. Therefore \(\bar X_0\to\X_{0,s}\) is the desingularization at \(P\).
	
	\emph{(iii)+(iv)}
	The curve \(\bar X_1\) is the projective Weierstrass cubic obtained from
	\eqref{eq:exceptional_curve_equation}. Its unique point at infinity is
	\[
	\tilde P=[0:1:1]\in\PP_k(1,2,3).
	\]
	This point is smooth, with local parameter
	\[
	\frac{x_1}{y_1}
	=
	\frac{\tilde\Pi\tilde x}{\tilde y}
	=
	\frac{\Pi x}{y}
	=
	s.
	\]
	On the chart \(D_+(\tilde x)\cap D_+(\tilde y)\), the branch \(r=0\) is therefore
	precisely the local branch of \(\bar X_1\) at \(\tilde P\). Since the point at
	infinity is the only point of the projective Weierstrass cubic with
	\(\tilde\Pi=0\), we have
	\[
	\bar X_0\cap\bar X_1=\{\tilde P\}.
	\]
	Moreover \(\tilde P\) is a smooth point of \(\bar X_1\).
	
	The projective Weierstrass cubic has arithmetic genus one. Since it has a unique
	smooth point at infinity, it cannot be reducible: every projective component
	would meet the line at infinity. Hence \(\bar X_1\) is integral of arithmetic genus
	one.
	
	\emph{(v)}
	The completed local ring computed above gives
	\[
	\widehat{\OO}_{\X,\tilde P}
	\cong
	\OO_K[[r,s]]/(rs-\Pi).
	\]
	Therefore \(\tilde P\) is an ordinary node of \(\X_s\), and its thickness is
	\[
	v_K(\Pi)=t.
	\]
	This proves the proposition.
\end{proof}

\begin{remark} \label{rem:t_max}
	The definition of $t_{\max}(x,y)$ and of the coefficients $\bar{a}_i$ in \eqref{eq:exceptional_curve_equation}, together with Lemma \ref{lem:t_max} (iii), immediately imply the following statement:
	\begin{quotation}
		There exists $i\in\{1,2,3,4,6\}$ with $\bar{a}_i\neq 0$ if and only if $t=t_{\max}(x,y)$. 
	\end{quotation}
	In particular, if $t<t_{\max}(x,y)$ then $\Xb_1$ is the cuspidal curve $y^2=x^3$, which is not semistable, and then $\X\to\X_0$ is not a resolution of the cusp $P$.
\end{remark}

\section{Existence of stable cusp resolution} \label{sec:existence_resolution}

In this section we give a proof of Theorem \ref{thm:cusp_resolution}.

\subsection{Optimal coordinates} \label{subsec:optimal_coordinates}

The modification $\X\to\X_0$ constructed in \S \ref{sec:blow-up} depends on the choice of an admissible coordinate system $(x,y)$ and on the parameter $t$, a positive rational number $\leq t_{\max}(x,y)$ (where $t_{\max}(x,y)$ is defined by \eqref{eq:t_def}). To prove Theorem \ref{thm:cusp_resolution} we have to find $(x,y)$ and $t$ such that the exceptional fiber of the blow-up  $\X\to\X_0$ is semistable (Corollary \ref{cor:cusp_resolution}). By Remark \ref{rem:t_max} we must choose $t=t_{\max}$. 

The following definition gives the right condition on the coordinate system $(x,y)$. 

\begin{definition} \label{def:optimal}
	An admissible coordinate system $(x,y)$ is called {\em optimal} if the rational value $t_{\max}(x,y)$ takes its maximal value, among all admissible coordinate systems after allowing arbitrary finite extensions of $K$.
\end{definition} 

By Remark \ref{rem:scaling}, every admissible
coordinate system can be made strictly admissible by multiplying \(x\), \(y\),
and \(f\) by units. This operation does not change \(t_{\max}\), nor the resulting
weighted blow-up up to the evident isomorphism. We may therefore always assume that admissible coordinates are, in fact, strictly admissible.

Theorem \ref{thm:cusp_resolution} follows immediately from the following two lemmata.

\begin{lemma} \label{lem:optimal_iff_resolution}
	If an admissible coordinate system $(x,y)$ is optimal then the blow-up morphism $\X((x,y),t)\to\X_0$ with parameter $t=t_{\max}(x,y)$ is the stable resolution of the cusp $P$.
\end{lemma}

\begin{lemma} \label{lem:optimal_cs_exists}
	After replacing $K$ by a finite extension, there exists an optimal coordinate system.  	
\end{lemma}

The rest of this section is concerned with proving these two lemmata.

\subsection{Proof of Lemma \ref{lem:optimal_iff_resolution}} \label{subsec:optimal_if_resolution}

Let us fix strictly admissible coordinates $(x,y)$, and set $t:=t_{\max}(x,y)$. After an extension of $K$, we may choose an element $\Pi\in K$ with $v_K(\Pi)=t$. Let $\X\to\X_0$ be the weighted  blow-up corresponding to the choice of $(x,y)$ and $t$, with exceptional fiber $\Xb_1$. We assume that $\X\to\X_0$ is not a stable resolution, i.e.\ that $\Xb_1$ is not semistable. To prove Lemma \ref{lem:optimal_iff_resolution} it then suffices to show that $(x,y)$ is not optimal.

\medskip
The proof of Proposition \ref{prop:cusp_resolution} (iii) revealed an explicit equation for the curve $\Xb_1$, obtained as follows.
Write
\[
  F:= \Pi^{-6}f(x,y) = \Pi^{-6}f(\Pi^2 x_1,\Pi^3 y_1) = \sum_{i,j} b_{i,j}x_1^i y_1^j
\]
as an element of the subring $\OO_K[x_1,y_1]\subset K[x,y]$ with generators $x_1:=\Pi^{-2}x$, $y_1:=\Pi^{-3}y$. Then $F$, as a polynomial in $x_1,y_1$, is primitive, and its reduction to the residue field $k$ is an affine Weierstrass equation for $\Xb_1$,
\begin{equation} \label{eq:Fb}
   \bar{F}=\sum_{i,j} \bar{b}_{i,j} x_1^iy_1^j = 0.
\end{equation}
In particular, $\Xb_1$ is a Weierstrass cubic. So our assumption means that $\Xb_1$ is cuspidal, i.e.\ it has a cusp as singularity.

Let $\delta,\epsilon,\zeta\in\OO_K$ be integral elements, to be determined later. Set
\begin{equation} \label{eq:delta_epsilon_zeta}
    \alpha:=\Pi^2\delta,\;\; \beta:=\Pi^3\epsilon,\;\; \gamma:=\Pi\zeta
\end{equation}
and let $g\in\GL_3(K)$ denote the matrix representing the coordinate change
\begin{equation} \label{eq:alha_beta_gamma_cs}
	x\leftarrow x+\alpha, \quad y\leftarrow y +\gamma x +\beta.
\end{equation}
It transforms the equation $f\in A$ for $\X_0\subset\AA_{\OO_K}^2$ into 
\begin{equation} \label{eq:f'}
	f':=\lexp{g}{f} = \sum_{i,j} A_{i,j} x^i y^j,
\end{equation}	
with $A_{i,j}\in\OO_K$. 

Let $(x',y'):=\lexp{g}{(x,y)}$ be the new coordinates, obtained after the coordinate change by $g$. We claim that 
\begin{equation} \label{eq:t_max_new}
  t_{\max}(x',y') \geq t.	
\end{equation}
To see this, write
\begin{equation} \label{eq:g^F}
   F' := \Pi^{-6}(f')(\Pi^2x_1,\Pi^3y_1)
     = \sum_{i,j} B_{i,j} x_1^i y_1^j, 
\end{equation}
with
\[
    B_{i,j} := \Pi^{2i+3j-6} A_{i,j} \in K.
\] 
Regarding $F$ as a polynomial in $x_1,y_1$ and using \eqref{eq:delta_epsilon_zeta} a straightforward calculation shows that
\begin{equation} \label{eq:g^F_delta_epsilon_gamma} 
    F' = F(x_1+\delta, y_1+\zeta x_1 + \epsilon).
\end{equation}
This shows that the $B_{i,j}\in \OO_K$ are integers. It follows that 
\begin{equation} \label{eq:v_K(A_i,j)}
    v_K(A_{i,j}) = v_K(B_{i,j}) + (6 -2i -3j)t \geq (6 - 2i -3j)t.
\end{equation}
By \eqref{eq:t_def} this implies \eqref{eq:t_max_new}. 

With \eqref{eq:t_max_new} proved, the weighted blow-up $\X'\to\X_0$ corresponding to the coordinates $(x',y')$ and the constant $t\leq t_{\max}(x',y')$ is well defined. On the affine patch $\Spec\OO_K[x_1,y_1]\subset\P$, $\X'$ is defined by the polynomial equation \eqref{eq:g^F}. So by \eqref{eq:g^F_delta_epsilon_gamma}, the exceptional fiber $\Xb_1'$ is the Weierstrass cubic defined by
\begin{equation} \label{eq:g^Fb}
   \bar{F}' = \bar{F}(x_1+\bar{\delta}, y_1+\bar{\zeta} x_1 + \bar{\epsilon}),
\end{equation}
where $\bar{\delta},\bar{\epsilon},\bar{\zeta}\in k$ are the residue classes of $\delta,\epsilon,\zeta$. Since we assume that the Weierstrass cubic $\Xb_1$ defined by \eqref{eq:Fb} is cuspidal, Proposition \ref{prop:RWNF} (ii) shows that we can choose $\delta,\epsilon,\zeta\in\OO_K$ such that \eqref{eq:g^Fb} is the standard cuspidal cubic $y_1^2-x_1^3$. In other words, for all $(i,j)$ with $2i+3j<6$ we have $v_K(B_{i,j})>0$. It follows that
the inequality in \eqref{eq:v_K(A_i,j)} is strict, for all $(i,j)$ with $2i+3j<6$. Therefore,
\[
    t_{\max}(x',y') > t=t_{\max}(x,y).
\]  
So we have shown that the coordinates $(x,y)$ are {\em not} optimal. This concludes the proof of Lemma \ref{lem:optimal_iff_resolution}.

\subsection{Inessential coordinate changes}

The proof of Lemma~\ref{lem:optimal_cs_exists} will maximize \(t_{\max}\) on
a restricted family of coordinate changes. We first justify this reduction. The
point is that some affine-linear changes do not affect \(t_{\max}\), while every
admissible coordinate system is, up to such changes, obtained from a fixed one
by a strict unipotent change.

\begin{definition}
	\label{def:inessential-coordinate-change}
	An inessential coordinate change is an affine-linear coordinate change of the form
	\[
	x'=\lambda x+\eta y,\qquad y'=\mu y,
	\]
	where
	\[
	\lambda,\mu\in\mathcal O_K^\times,\qquad \eta\in\mathcal O_K.
	\]
	Two admissible coordinate systems are called equivalent if one is obtained from the other by an inessential coordinate change.
\end{definition}

\begin{lemma}
	\label{lem:tmax-invariant-inessential}
	For admissible coordinates $(x,y)$ the value \(t_{\max}(x,y)\) is invariant under inessential coordinate changes.
\end{lemma}

\begin{proof}
	Unit rescalings of \(x\) and \(y\) only multiply coefficients by units, and so does not affect $t_{\max}(x,y)$. It remains
	to consider \(x\leftarrow x+\eta y,\;y\leftarrow y\). Since
	\[
	(x+\eta y)^iy^j
	=
	x^iy^j+\text{terms of strictly larger \((2,3)\)-weight},
	\]
	such a coordinate change does not affect $t_{\max}(x,y)$ either.
\end{proof}

\begin{definition}
	\label{def:strict-unipotent-change}
	A strict unipotent coordinate change is a coordinate change of the form
	\[
	x \leftarrow x+\alpha,\qquad y\leftarrow y+\gamma x+\beta,
	\]
	with \(\alpha,\beta,\gamma\in\mathfrak m_K\).
\end{definition}

\begin{lemma}
	\label{lem:admissible-slice}
	Every admissible coordinate system is equivalent to one obtained from a fixed
	admissible coordinate system by a strict unipotent coordinate change.
\end{lemma}

\begin{proof}
	Write the transition between the two coordinate systems as
	\[
	x'=ax+by+\alpha,\qquad y'=cx+dy+\beta .
	\]
	Since both systems are centered at \(P\), we have \(\alpha,\beta\in\m_K\).
	Comparing the orders of \(\bar x'\) and \(\bar y'\) on the cusp gives
	\[
	a,d\in\OO_K^\times,\qquad c\in\m_K .
	\]
	Set
	\[
	\mu:=d,\qquad \gamma:=c/d,\qquad
	\eta:=b,\qquad \lambda:=a-bc/d .
	\]
	Then \(\gamma\in\m_K\) and \(\lambda,\mu\in\OO_K^\times\). Put
	\[
	\beta_0:=\beta/\mu,\qquad
	\alpha_0:=(\alpha-\eta\beta_0)/\lambda .
	\]
	Then \(\alpha_0,\beta_0,\gamma\in\m_K\), and
	\[
	x'=\lambda(x+\alpha_0)+\eta(y+\gamma x+\beta_0),
	\qquad
	y'=\mu(y+\gamma x+\beta_0).
	\]
	Thus \((x',y')\) is obtained from the strict unipotent change
	\[
	x\leftarrow x+\alpha_0,\qquad
	y\leftarrow y+\gamma x+\beta_0
	\]
	by an inessential coordinate change.
\end{proof}

\subsection{Existence of optimal coordinates}
\label{subsec:existence_of_optimal_coordinates}

We now prove Lemma \ref{lem:optimal_cs_exists}, using the methods from \cite[\S 5]{SternWewers}.

\medskip
Let $U\subset\GL_{3,K}$ denote the linear algebraic group of upper unipotent matrices
\begin{equation} \label{eq:g_matrix}
   g = \begin{pmatrix}
   	  1 & \alpha & \beta \\ 0 & 1 & \gamma \\ 0 & 0 & 1
   \end{pmatrix}.
\end{equation}
In the following, we consider $\alpha,\beta,\gamma$ as variables, so that $U=\Spec K[\alpha,\beta,\gamma]\cong\AA^3_K$. 

Let us fix one admissible coordinate system $(x,y)$, giving rise to the polynomial equation $f\in\OO_K[x,y]$ for $\X_0\subset\AA_{\OO_K}^2$. As in \eqref{eq:f'}, we write
\begin{equation} \label{eq:f^g}
     f^g =  \sum_{i,j} A_{i,j}(g)x^iy^j
\end{equation}
for a point $g\in U$, but now we regard $A_{i,j}\in K[\alpha,\beta,\gamma]$ as a regular function on $U$, i.e.\ as a polynomial in $\alpha,\beta,\gamma$.  

Let $U^\rig$ denote the rigid analytic space over $K$ associated to $U$, see \cite[\S 5.4]{BoschRigid}. The underlying set of $U^\rig$ is simply the set of closed points of the scheme $U$. Let $U_0\subset U^\rig$ denote the unit disk, i.e.\ the affinoid subdomain defined by the conditions
\[
   v_K(\alpha) \geq 0,\; v_K(\beta)\geq 0,\; v_K(\gamma)\geq 0. 
\]

Given a point $g\in U^\rig$, the residue field $L:=K(g)$ is a finite extension of $K$, and we may then consider $g\in U(L)$ as an $L$-rational point of $U$, and $(x',y'):=\lexp{g}{(x,y)}$ as coordinates on $\AA_L^2$.  We have $g\in U_0$ if and only if the matrix entries  $\alpha,\beta,\gamma\in\OO_L$ of $g$ are integral. 

For an element $A\in K[\alpha,\beta,\gamma]$, we set $v_g(A) := v_L(A(g))$. We define a function $t_{\max}:U_0\to\RR\cup\{\infty\}$ by 
\[
	t_{\max}(g) := \min_{2i+3j<6} \frac{v_g(A_{i,j})}{6-2i-3j}.  
\]
So for $g\in U_0$ we have 
\[
   t_{\max}(g)=t_{\max}(\lexp{g}{(x,y)}), 
\]
where $(x',y')=\lexp{g}{(x,y)}$ is the new coordinate system over $L:=K(g)$, obtain by the coordinate change $g$, and $t_{\max}(x',y')$ is defined by \eqref{eq:t_def}. Note that we cannot exclude the possibility that $t_{\max}(g)=\infty$. For instance, the point on $\AA_L^2$ with coordinates $x'=0,y'=0$ may be a singular point of the generic fiber  $X_L$ (but may also not lie on $X_L$ at all). Therefore, the conclusion of Lemma \ref{lem:admissible} (ii) is not available.
	
By our assumptions on $\X_0$ there exists a primitive polynomial $h\in\OO_K[x,y]$ such that the open subset $D(h)\subset\X_0$ is a neighborhood of $P$ and such that $D(h)\backslash\{P\}$ is smooth over $\OO_K$. In particular, $D(\bar{h})\subset \X_{0,s}$ is an open neighborhood of $P$ on which $P$ is the only singular point. Let $V\subset U_0$ denote the rational subdomain defined by the condition
\begin{equation}
  v_g(C_{0,0}) = 0, \qquad \text{where}\;\;\lexp{g}{h}=\sum_{i,j} C_{i,j}(g)x^iy^j.
\end{equation}
Note that $1_U\in V$ because $\bar{h}(P)\neq 0$.

\begin{lemma} \label{lem:t_max_admissible}
  For $g\in V$ we have  $t_{\max}(g)<\infty$. Moreover, $t_{\max}(g)>0$ if and only if the coordinates $\lexp{g}{(x,y)}$ are admissible.
\end{lemma}

\begin{proof}
	Put
	\[
	F:=\lexp{g}{f}=\sum A_{i,j}(g)x^iy^j .
	\]
	Since \(g\in U_0\), all coefficients of \(F\) are integral. 
	Let \(Q\) be the
	special-fiber point which is the origin in the coordinates \(\lexp{g}{(x,y)}\).
	Since \(g\in V\), we have \(Q\in D(\bar h)\). 
	
	If all \(A_{i,j}(g)\) with \(2i+3j<6\) vanished, then the origin in the
	coordinates \(\lexp{g}{(x,y)}\) would be a singular point of the generic fiber, specializing to $Q$, contradicting the smoothness
	of \(D(h)\setminus\{P\}\). Hence \(t_{\max}(g)<\infty\).
	
	Assume first that \(t_{\max}(g)>0\). Then the coefficients of \(F\) of
	\((2,3)\)-weight \(<6\) vanish modulo \(\mathfrak m\). In particular, the
	constant and linear terms vanish, so \(Q\) is a singular point of \(\X_{0,s}\).
	But \(P\) is the only singular point of \(\X_{0,s}\) in \(D(\bar h)\). Hence
	\(Q=P\), so the new coordinates are centered at \(P\).
	
	It remains to check the weight-six terms. Since \(P\) is a cusp, the multiplicity
	of \(\X_{0,s}\) at \(P\) is \(2\). As the coefficients of \(x^2\) and \(xy\) vanish,
	the quadratic tangent cone is necessarily
	\[
	\bar A_{0,2}y^2
	\]
	with \(\bar A_{0,2}\ne0\). Thus \(v(A_{0,2}(g))=0\). The tangent line is
	therefore \(y=0\). For a cusp, the tangent line has intersection multiplicity
	\(3\) with the curve. Hence \(F(x,0)\) has nonzero cubic term modulo
	\(\mathfrak m\), i.e.
	\[
	\bar A_{3,0}\ne0.
	\]
	Thus \(v(A_{3,0}(g))=0\), and the coordinates are admissible.
	
	Conversely, if \(\lexp{g}{(x,y)}\) is admissible, then by definition all
	coefficients of weight \(<6\) have positive valuation. Hence
	\(t_{\max}(g)>0\).
\end{proof}  	
  	
We can now finish the proof of Lemma \ref{lem:optimal_cs_exists}. 	
By \cite[Lemma 5.1]{SternWewers}, the function $t_{\max}$ attains a maximum at some point $g_{\max}\in V$. Since \(1_U\in V\) and the fixed coordinates are admissible, the maximum is at least
\(t_{\max}(1)>0\). Therefore, by Lemma~\ref{lem:t_max_admissible}, $g_{\max}$ corresponds to an admissible coordinate system.

By Lemmas~\ref{lem:tmax-invariant-inessential} and
\ref{lem:admissible-slice}, every admissible coordinate system is equivalent to one obtained from the fixed system by a strict unipotent change, and \(t_{\max}\) is unchanged under this equivalence. Hence the maximum on \(V\) is the maximum among all admissible coordinate systems. We conclude that the coordinates corresponding to $g_{\max}$ are optimal. This concludes the proof of Lemma \ref{lem:optimal_cs_exists} and, as a consequence, of Theorem \ref{thm:cusp_resolution}. 

\section{Rigidification} \label{sec:rigidification}

The construction of the blow-up $\X\to\X_0$ used in the proof of Theorem \ref{thm:cusp_resolution} is very explicit, but this does not yet give us a practical algorithm to compute the resolution. The problem is the proof of Lemma \ref{lem:optimal_cs_exists}, where we use a version of the nonarchimedean maximum principle which seems hard to implement. In this section we give an alternative criterion for an admissible coordinate system $(x,y)$ to lead to a resolution. In \S \ref{sec:implementation} we then show that this leads to a practical algorithm to find such a coordinate system. 

\subsection{Rigidified coordinates} 

Set 
\begin{equation} \label{eq:I_def}
	I := \begin{cases}
		\{(0,0), (1,0), (2,0) \}, & p=2, \\
		\{ (0,0), (0,1), (1,1)\}, & p\neq 2.
	\end{cases}
\end{equation}	
Let $(x,y)$ be a strictly admissible coordinate system and
\[
f(x,y) = \sum_{i,j} a_{i,j} x^iy^j
\]
the corresponding equation of the affine plane model $\X_0$. We consider the following two conditions on $(x,y)$.

\begin{cond} \label{cond:coeffs}
	\begin{enumerate}[(a)]
		\item 
		We have $a_{i,j}=0$, for all $(i,j)\in I$.
		\item 
		We have
		\[
		v_K(a_{i,j}) > (6-2i-3j)\cdot t_{\max}(x,y),
		\] 
		for all $(i,j)\in I$. 
	\end{enumerate}  	
\end{cond}

It is clear that (a) implies (b). Condition \ref{cond:coeffs} (b) means that the minimum in the definition \eqref{eq:t_def} of $t_{\max}(x,y)$ is {\em not} attained for any of the pairs $(i,j)\in I$. Using (b) we can formulate an alternative to Lemma \ref{lem:optimal_iff_resolution}.

\begin{proposition} \label{prop:cond_b_is_sufficient}
	Let $\X\to\X_0$ be the blow-up corresponding to the strictly admissible coordinate system $(x,y)$ with parameter $t=t_{\max}(x,y)$. If Condition \ref{cond:coeffs} (b) holds, then $\X\to\X_0$ is the stable resolution of the cusp $P$.	
\end{proposition}

\begin{proof}
	Recall that the exceptional divisor $\Xb_1$ of $\X\to\X_0$ is the Weierstrass cubic given by \eqref{eq:exceptional_curve_equation}. It follows from Condition \ref{cond:coeffs} (b) that, with the notation of the proof of Proposition \ref{prop:cusp_resolution} (iii), $\bar{b}_{i,j}=0$ for $(i,j)\in I$ and hence $\bar{a}_i=0$ for $i\in\{2,4,6\}$ if $p=2$ and $i\in\{1,3,6\}$ if $p\neq 2$. In terms of Definition \ref{def:RWNF} this means that $\Xb_1$ is a cubic in {\em weak rigidified Weierstrass normal form}. But since $t=t_{\max}(x,y)$, there exists at least one index $i\in\{1,2,3,4,6\}$ such that $\bar{a}_i\neq 0$ (Remark \ref{rem:t_max}). By Proposition \ref{prop:RWNF} (ii) it follows that $\Xb_1$ is semistable and hence $\X\to\X_0$ is the stable resolution.
\end{proof}

How does Proposition \ref{prop:cond_b_is_sufficient} help us to find a
resolution in practice? The idea is to look for a coordinate system satisfying
the stronger Condition \ref{cond:coeffs} (a). For the coordinate changes
considered below, this amounts to solving a system of algebraic equations over
\(K\). Since Condition \ref{cond:coeffs} (b) is an open valuation condition, a
sufficiently accurate approximation to a solution of this system still satisfies
Condition \ref{cond:coeffs} (b). Proposition \ref{prop:cond_b_is_sufficient}
then applies.

The next step is to show that the sufficient criterion above is not empty: after a finite
extension of \(K\), there exist admissible coordinates satisfying the stronger
rigidifying condition.

\subsection{Existence of rigidified coordinates} \label{subsec:rigidified}

We retain the notation used in \S \ref{subsec:existence_of_optimal_coordinates}. In particular, we fix one strictly admissible coordinate system $(x,y)$.

For $s\in\QQ$, $s\geq 0$, let $U_s\subset U^\rig$ denote the affinoid subdomain defined by the inequalities
\[
   v_L(\alpha)\geq 2s,\;\; v_L(\beta)\geq 3s,\;\; v_L(\gamma)\geq s,
\]	
and $U_s^+\subset U_s$ the residue class defined by the strict inequalities. One easily checks that both subsets are actually {\em subgroups}.

If $L/K$ is a finite extension, and $(x',y')$ is another strictly admissible coordinate system, defined over $L$, then $(x',y')=\lexp{g}{(x,y)}$, for a unique $g\in U^\rig$ with $K(g)\subset L$. We say that $g$ is {\em perfect} if $(x',y')=\lexp{g}{(x,y)}$ leads to the stable resolution.

\begin{proposition} \label{prop:perfect}
  Let $g\in U_0^+$ be perfect. Then the subset of $U_0^+$ of all perfect elements is the left coset
  \[
       U_t\cdot g \subset U^\rig,
  \] 
  where $t=t_{\max}(\lexp{g}{(x,y)})$.
\end{proposition}

\begin{proof}	
Let $g'\in U_0^+$ and set $h:=g'g^{-1}$. We choose an extension $L/K$ such that $g,g',h\in U(L)$ and write $h=(\alpha,\beta,\gamma)$, with $\alpha,\beta,\gamma\in\m_L$. Put  $(x_g,y_g)=\lexp{g}{(x,y)}$, $(x_{g'},y_{g'})=\lexp{{g'}}{(x,y)}$. Then 
\begin{equation}
  x_{g'} = x_g +\alpha,\quad y_{g'} = y_g + \gamma x_g + \beta.	
\end{equation}
After enlarging $L$ we may also assume that there exists an element $\Pi\in L$ with $v_L(\Pi)=t$. 

Let $\X\to\X_0$ be the blow-up constructed using the coordinate system $(x_g,y_g)$ and the parameter $t$. Our assumption that $g$ is perfect means that $\X\to\X_0$ is the stable resolution of the cusp $P$. By construction, the affine chart of $\X$ containing the standard affine chart of the exceptional component $\Xb_1$ is $\Spec \OO_L[x_1,y_1]/(F)$, where 
\[
    x_1 := \Pi^{-2}x_g,\;\; y_1:= \Pi^{-3}y_g,
\]
and where $F=\Pi^{-6}(\lexp{g}{f})\in\OO_L[x_1,y_1]$ (see the proof of Proposition \ref{prop:cusp_resolution}.)

Assume that $g'$ is also perfect. Then the blow-up $\X'\to\X_0$ defined using the coordinates $(x_{g'},y_{g'})$ and the constant $t'=t_{\max}(\lexp{{g'}}{(x,y)})$ is also a stable resolution. The uniqueness of the stable resolution shows that there exists a (unique) isomorphism $\X\iso\X'$ extending the identity on the generic fiber. It follows that $t=t'$ (by Proposition \ref{prop:cusp_resolution} (v), this number is equal to the thickness of the node $\tilde{P}$, which is an invariant under isomorphism). Moreover, as elements of $K[x,y]$, the coordinates 
\[
 x'_1 := \Pi^{-2}x_{g'} = x_1 + \Pi^{-2}\alpha,\quad y'_1:= \Pi^{-3}y_{g'} = y_1 + \Pi^{-1}\gamma x_1 + \Pi^{-3}\beta,
\]
must be regular on the affine chart
\(\Spec \OO_L[x_1,y_1]/(F)\), in particular at the generic point of the
exceptional component. This implies
\[
    v_L(\alpha) \geq 2t,\;\; v_L(\beta) \geq 3t,\;\; v_L(\gamma)\geq t,
\]
and so $g'\in U_t\cdot g$. 

Conversely, assume that $g'\in U_t\cdot g$. Then the proof of Lemma \ref{lem:optimal_iff_resolution} shows that $t\leq t_{\max}(x_{g'},y_{g'})$. Moreover, if $\X'\to\X_0$ denotes the blow-up constructed using the coordinates $(x_{g'},y_{g'})$ and the constant $t$, then the exceptional divisor $\Xb_1'$ of $\X'\to\X_0$ is isomorphic to $\Xb_1$. Indeed, if $\Xb_1$ is given by the Weierstrass equation $\bar{F}(x,y)=0$, then $\Xb_1'$ is given by the equation
\[
      \bar{F}' = \bar{F}(x+\bar{\delta}, y+\bar{\zeta} x + \bar{\epsilon}),
\] 
where $\bar{\delta},\bar{\epsilon},\bar{\zeta}\in k$ are the residue classes of 
\[
   \delta :=\Pi^{-2}\alpha, \;\; \epsilon :=\Pi^{-3}\beta,\;\; \zeta:=\Pi^{-1}\gamma,
\]
see \eqref{eq:g^Fb}. This means that $\Xb_1'$ is also semistable and hence $\X'\to\X_0$ is a stable resolution and $g'$ is perfect. 
\end{proof}	

From now on, $t$ always denotes the value $t=t_{\max}(\lexp{g}{(x,y)})$, where $g\in U_0^+$ is perfect, and we set
\[
    U^{\rm perf} := U_t\cdot g\subset U^\rig.
\]
This is an affinoid subdomain, independent of the choice of $g$.

We say that the ground field $K$ is {\em adapted} if there exists $g\in U^{\rm perf}(K)$ and an element $\Pi\in K$ with $v_K(\Pi)=t$. This is true after replacing $K$ by a suitable finite extension, and then $\Pi$ will always denote such an element. A general point $g'\in U^{\rm perf}$ is then of the form $g'=h\cdot g$, where $h$ is the unipotent matrix with entries
\begin{equation}
    \alpha= \Pi^2\delta,\;\; \beta = \Pi^3\epsilon,\;\; \gamma=\Pi\zeta,
\end{equation}
with $\delta,\epsilon,\zeta\in\OO_L$, for some  finite extension $L/K$, depending on $g'$. This gives a presentation of $U^{\rm perf}$ as a closed rigid polydisk,
\[
   U^{\rm perf} = {\rm Spm}\,K\langle \delta,\epsilon,\zeta\rangle,
\] 
with parameters $\delta,\epsilon,\zeta$. Equivalently, we may regard 
\(U^{\rm perf}\) as the tube of the special fiber in the
algebraic $\OO_K$-model
\[
 \U^{\rm perf} :=\Spec\OO_K[\delta,\epsilon,\zeta]
\]
of $U$. We shall identify the affinoid reduction of $U^{\rm perf}$ with
\[
\bar U^{\rm perf}=\Spec k[\delta,\epsilon,\zeta].
\]
Note that such a presentation exists only if $K$ is adapted, and depends on the choice of the element $g$. 

As in \eqref{eq:f'}, write 
\[
	f^g =  \sum_{i,j} A_{i,j}(g)x_g^iy_g^j,
\]
with $A_{i,j}\in K[\alpha,\beta,\gamma]$ and $(x_g,y_g)=\lexp{g}{(x,y)}$. If $g\in U^{\rm perf}$, then
\[
F := \Pi^{-6} f^g = \sum_{i,j} B_{i,j}(g) x_1^iy_1^j
\]
with
\[
B_{i,j} := \Pi^{2i+3j-6} A_{i,j} \in K[\alpha,\beta,\gamma],
\] 
is a polynomial in $x_1:=\Pi^{-2}x_g$, $y_1:=\Pi^{-3}y_g$ with integral coefficients, and is the defining equation for the resolution $\X\to\X_0$ corresponding to the coordinates $(x_g,y_g)$, on the affine chart $\Spec\OO_K[x_1,y_1]$. 
In particular, the exceptional fiber $\Xb_1$ is given by the Weierstrass equation
\begin{equation} \label{eq:Fb_new}
  \bar{F} = \sum_{i,j} \bar{B}_{i,j}(\bar{g}) x^iy^j,
\end{equation}
where $\bar{B}_{i,j}\in k[\delta,\epsilon,\zeta]$ denotes the reduction of $B_{i,j}$ and $\bar{g}\in\bar{U}^{\rm perf}$ the  specialization of $g$. See the proof of Lemma \ref{lem:optimal_iff_resolution}. 

\medskip
We consider the closed subscheme $S\subset U=\Spec K[\alpha,\beta,\gamma]$ defined by the three equations
\begin{equation} \label{eq:A_equations}
	A_{i,j}(\alpha,\beta,\gamma) = 0, \quad (i,j) \in I,
\end{equation}
and where $I$ is defined by \eqref{eq:I_def}. 
 
Let $S_\et\subset S$ be the locus where the morphism
$S\to \Spec K$ is \'etale. Since $S_\et\to \Spec K$ is
\'etale and of finite type, it is finite; via $S\subset U$, we may therefore
view it as a finite closed subscheme of $U$. In
\S\ref{subsec:resolve_cusp} this finite scheme is computed explicitly by
saturating the defining ideal of $S$ with respect to the relevant Jacobian
determinant.

We also consider the subset $S^{\rm adm}:=S^\rig\cap U_0^+\subset S^\rig$. We claim that 
\begin{equation} \label{eq:S^adm_in_U^perf}
	S^{\rm adm} \subset U^{\rm perf}.
\end{equation}
Indeed, let \(g\in S^{\rm adm}\). Then the coordinate system
\(\lexp{g}{(x,y)}\) satisfies Condition \ref{cond:coeffs} (a), hence also
Condition \ref{cond:coeffs} (b). By Proposition
\ref{prop:cond_b_is_sufficient}, \(g\) is perfect. The preceding proposition
therefore gives \(g\in U^{\rm perf}\).

Finally, we consider the intersection
\[
  S_\et^{\rm adm}:=S_\et\cap S^{\rm adm}.
\]
As $S_\et$ is a finite union of closed points of $U$, the same is true for $S_\et^{\rm adm}$.

\begin{lemma} \label{lem:S_et^adm}
	$S_\et^{\rm adm}$ is nonempty.
\end{lemma}

\begin{proof}
It suffices to prove nonemptiness after a finite extension of \(K\).
We may therefore assume that $K$ is adapted. 

Define, using \eqref{eq:S^adm_in_U^perf}, the closed subscheme $\Sc\subset\U^{\rm perf}$ by the equations
\[
B_{i,j}(\delta,\epsilon,\zeta) = 0, \quad (i,j) \in I.
\]
Its generic fiber is $\Sc_K=S^{\rm adm}$, and its special fiber 
$\bar{S}\subset\bar{U}^{\rm perf}$ is defined by 
\begin{equation} \label{eq:Sb_equation}
	\bar{B}_{i,j}(\delta,\epsilon,\zeta) = 0, \quad (i,j) \in I.
\end{equation}
Let $\Sc'\subset\Sc$ be the open subscheme defined by 
\[
B_{i,j}(\delta,\epsilon,\zeta) \neq 0, \quad \text{for $(i,j)=(0,1)$ if $p=2$ and $(i,j)=(1,0)$ otherwise.}
\]
Let $\bar{S}'\subset\bar{S}\subset\Ub$ denote the special fibers of $\Sc$ and $\Sc'$. By construction, $\bar{S}'\subset\bar{S}$ is defined by the open condition
\begin{equation} \label{eq:Sb'_equation}
	\bar{B}_{i,j}(\delta,\epsilon,\zeta) \neq 0, \quad \text{for $(i,j)=(0,1)$ if $p=2$ and $(i,j)=(1,0)$ otherwise.}
\end{equation}
By Definition \ref{def:RWNF}, the validity of \eqref{eq:Sb_equation} means that the equation \eqref{eq:Fb_new} for $\Xb_1$ is in weak rigidified Weierstrass normal form if $g\in S^{\rm adm}$. Moreover, the additional condition \eqref{eq:Sb'_equation} - which holds if $\bar{g}\in\bar{S}'$ - means that it is in strong rigidified Weierstrass normal form. 

With this preparation, Proposition \ref{prop:srwnf_is_etale} shows that \(\bar S'\) is
nonempty and finite étale over \(k\). As \'etaleness is an open condition, $\Sc'$ is \'etale over $\OO_K$ in a neighborhood of $\bar{S}'$. Choose a point
\(\bar g\in \bar S'(k)\). Since \(\OO_K\) is strictly henselian, \(\bar g\) lifts uniquely to a point $g\in (\Sc')_K\subset S^{\rm adm}$. As a point on $S$, it lies in $S_\et$ because \'etaleness commutes with arbitrary base change. Hence it lies in \(S_\et^{\rm adm}\).
\end{proof}

Lemma \ref{lem:S_et^adm} proves the required existence statement on the
level of the rigidifying scheme. Translating this back into coordinates gives
exactly the promised conclusion: after a finite extension, there is a strict
unipotent coordinate change satisfying the equations \eqref{eq:A_equations},
and hence Condition \ref{cond:coeffs} (a). The following corollary records this
immediate consequence.

\begin{corollary} \label{cor:existence_of_solutions}
  There exist a finite separable extension $L/K$ and a solution $(\alpha,\beta,\gamma)$ in $\m_L^3$ of the system of equations \eqref{eq:A_equations}. Given such a solution, and after application of the corresponding strict unipotent coordinate change, the weighted blow-up from \S \ref{subsec:blow-up}, with parameter $t=t_{\max}(x,y)$, is the stable resolution of the cusp $P$.
\end{corollary}

\begin{definition} \label{def:strongly_rigidified}
We call a point \(g\in S_\et^{\rm adm}\) strongly rigidified if, after adjoining
an element \(\Pi\) with
\[
   v(\Pi)=t_{\max}(\lexp{g}{(x,y)}),
\]
the exceptional Weierstrass cubic attached to the coordinates
\(\lexp{g}{(x,y)}\) is in strong rigidified Weierstrass normal form.
\end{definition}

For instance, the point $g$ produced in the proof of Lemma \ref{lem:S_et^adm} has this property. The following lemma will be used in the next subsection to control the field of
definition of  a strongly rigidified point $g\in S_\et^{\rm adm}$.

\begin{lemma} \label{lem:unique_hensel_lift}
  Let $g_1,g_2\in S_\et^{\rm adm}$. Assume that $g_1$ is strongly rigidified, and that $g_1g_2^{-1}\in U_t^+$, where \(t:=t_{\max}(\lexp{{g_1}}{(x,y)})\). Then $g_1=g_2$.
\end{lemma}	

\begin{proof}
The assumptions imply that $g_1,g_2$ specialize to the same point $\bar{g}\in\bar{S}'$. The proof of Lemma \ref{lem:S_et^adm} shows that $\Sc'\to\Spec\OO_K$ is \'etale over $\bar{g}$. From the uniqueness of the Hensel lift it follows that $g_1=g_2$.
\end{proof}	

\subsection{The monodromy action} \label{subsec:monodromy}

We will see in the next section that the previous result allows us to compute a stable resolution explicitly. Here we want to show that this method automatically returns the {\em minimal} field extension $L/K$ over which a stable resolution exists. An explicit example is discussed in \cite[\S 4.1, Proposition 4.1]{SSW}

We first recall the definition and the main properties of the monodromy action.

\begin{proposition}\label{prop:monodromy}
	Let \((\X_0,P)\) be a smoothing of a plane curve singularity, as in \S \ref{subsec:stable_resolution}.
	Let $L/K$ be a finite separable extension over which the stable resolution of $P$ exists. Let
	\[
	    \X\to\X_{0,L}
	\]
	be the stable resolution, defined over $L$, with tail $(E,\{\tilde{P}_1,\ldots,\tilde{P}_r\})$.
	\begin{enumerate}
	\item 
	  Assume that $L/K$ is Galois and set $\Gamma:=\Gal(L/K)$. Then the tautological action of $\Gamma$ on $X_L$ extends uniquely to an action on $\X$; its restriction to the tail $E$ is a $k$-linear action 
	  \[
	      \phi:\Gamma \to\Aut_k(E,\{\tilde{P}_1,\ldots,\tilde{P}_r\}),
	  \] 
	  called the {\em monodromy action}.
	 \item 
	   The fixed field $L^{\min}:=L^{\ker(\phi)}$ is the minimal extension over which the stable resolution exists. In particular, $L=L^{\min}$ if and only if the monodromy action is faithful.
	\end{enumerate}
\end{proposition}

\begin{proof}
  Claim (i) follows immediately from the uniqueness of the stable resolution. See e.g.\ \cite[Corollary 1.6]{Temkin2010}. For Part (ii) it suffices to show that the quotient scheme
  \[
      \X/\ker(\phi)
  \]
  is the stable resolution of $\X_{0,L^{\min}}$. This descent step is standard; e.g.\ it is the same argument as in
  \cite[Proposition 1.5]{WeilRep}, applied to a sufficiently small invariant
  neighborhood of the exceptional divisor. 
\end{proof}

We retain the notation from the previous subsection. In particular, we fix strictly admissible coordinates $(x,y)$. We do not assume that $(x,y)$ is optimal, nor do we assume that the ground field $K$ is adapted; instead, it will be fixed throughout.  

Let $g\in S_\et^{\rm adm}$ be strongly rigidified (Definition \ref{def:strongly_rigidified}). Let $L_0:=K(g)$ denote the residue field of $g$. Explicitly, $L_0=K(\alpha,\beta,\gamma)$, where $(\alpha,\beta,\gamma)$ is a solution to the system of equations \eqref{eq:A_equations} such that $\alpha,\beta,\gamma\in\m_{L_0}$. We identify $g$ with the unipotent matrix with entries $\alpha,\beta,\gamma$. By Proposition \ref{prop:cond_b_is_sufficient}, the coordinates
\((x_g,y_g)=\lexp{g}{(x,y)}\) give the stable resolution; in other words, \(g\) is perfect.

Put $t:=t_{\max}(x_g,y_g)$, and let 
\[
     e:= \min\{ n\geq 1 \mid n\cdot t \in v_{L_0}(L_0^\times)\}
\] 
be the smallest integer such that $e\cdot t$ lies in the value group of $L_0$.

The rest of this section is concerned with the proof of the following theorem. 

\begin{theorem} \label{thm:L_is_minimal_extension}
We use the notation and assumptions from above. In particular, $g\in S_\et^{\rm adm}$ is strongly rigidified. 	
\begin{enumerate}
\item 
  The positive integer $e$ is invertible in the residue field $k$.
\item 
  Let $L/L_0$ be the (unique) totally and tamely ramified extension $L/L_0$ of degree $e$. Then $L/K$ is the minimal extension over which a stable resolution of the cusp $P$ exists.
\end{enumerate}	
\end{theorem}

For the proof of the theorem, we let $L/L_0$ be some totally ramified extension of degree $e$.\footnote{Existence is easy: if ${\rm char}(K)=0$, we can adjoin an $e$th root of a uniformizer. If ${\rm char}(K)=p>0$, take a suitable Artin-Schreier extension for the $p$-part of $e$.} Once we have proved (i), it follows that $L/L_0$ is tamely ramified and unique.  

There exists an element $\Pi\in L$ with $v_{L}(\Pi)=t$. By assumption, the blow-up
\[
   \X\to\X_{0,L}
\]	
corresponding to the coordinates $(x_g,y_g)$ and the choice of $\Pi$ is a resolution of the cusp $P$. Since the stable resolution is unique, up to unique isomorphism, and $t$ is the thickness of the node $\tilde{P}$ (Proposition \ref{prop:cusp_resolution} (v)), the existence of an element $\Pi$ with $v_L(\Pi)=t$ is a necessary condition for the existence of a stable resolution. This implies that $L/L_0$ is the minimal extension over which a resolution of $(\X_{0,L_0},P)$ exists. 

Let $\tilde{L}/K$ be the Galois closure of $L/K$, and consider the  monodromy action 
\[
   \phi:\Gamma:=\Gal(\tilde{L}/K) \to\Aut_k(\Xb_1),
\]
where $\Xb_1$ is the tail of the resolution. Let $\phi_0$ denote the restriction of $\phi$ to the subgroup $\Gal(\tilde{L}/L_0)$. 
By the preceding argument, Proposition \ref{prop:monodromy} (ii) implies 
\begin{equation} \label{eq:ker_phi_0}
  \ker(\phi_0) = \Gal(\tilde{L}/L).
\end{equation} 
It follows that $L/L_0$ is Galois and that $\phi_0$ induces a faithful monodromy representation
\[
    \bar{\phi}_0:\Gal(L/L_0)\inj \Aut_k(\Xb_1).
\]

Recall that the choice of the coordinates $(x_g,y_g)$ and the element $\Pi$ gives an explicit presentation of $\Xb_1$ as a projective plane cubic in Weierstrass normal form.\footnote{The particular choice of $g$ also forces $\Xb_1$ to be in strong rigidified Weierstrass normal form, but this fact is not directly used in the 
proof.} So the automorphism group of $\Xb_1$ is a finite subgroup of the Weierstrass group $W$,
\[
     \Aut_k(\Xb_1) \subset W\subset \PGL_3(k),
\]
see Definition \ref{def:W-group} and Corollary \ref{cor:wnf_stabilizer}.

\begin{lemma} \label{lem:monodromy_explicit}
  The monodromy action of $\Gamma:=\Gal(\tilde{L}/K)$ on the tail $\Xb_1$ of the resolution is given by
  \[
      \phi:\Gamma\to\Aut_k(\Xb_1), \quad
      \phi(\sigma) = \overline{T\cdot \lexp{\sigma}{T}^{-1}},
  \]	
  where 
  \begin{equation} \label{eq:T_matrix}
  	T := \begin{pmatrix} 
  		1 & & \\ & \Pi^2 & \\ & & \Pi^3
  	\end{pmatrix} \cdot g =
  	\begin{pmatrix} 
  		1 & \alpha & \beta \\ 0 & \Pi^2 & \Pi^2\gamma \\ 0 & 0  & \Pi^3
  	\end{pmatrix}.
  \end{equation}
\end{lemma}

\begin{proof}
The matrix $T$ is the base change matrix from our original admissible coordinates $(x,y)$, which are defined over $K$, to the coordinates $(x_1,y_1)$ of the affine patch $D_+(\tilde{\Pi})$ of the blow-up $\X$ constructed in \S \ref{subsec:blow-up}. From there, the lemma follows from the definition of the monodromy action via a direct computation. 
\end{proof}

Let $\sigma\in\Gal(L/L_0)$, so that $\lexp{\sigma}{g}=g$. Then the lemma shows that
\[
   \bar{\phi}_0(\sigma) = \begin{pmatrix}
      1 &  & \\ & \chi(\sigma)^{-2} & \\ & & \chi(\sigma)^{-3}
   \end{pmatrix},
\]
where $\chi:\Gamma\to k^\times$ is the tame character defined by
\[
   \chi(\sigma) := \overline{\lexp{\sigma}{\Pi}/\Pi}.
\]
As $\bar{\phi}_0$ is injective, the restriction of $\chi$ to $\Gal(L/L_0)$ is injective, too. So $\chi$ has order $e$. It follows that $e$ is invertible in $k$. This proves Part (i) of the theorem. 

The remaining claims of the Theorem follow from the next lemma and Proposition \ref{prop:monodromy} (ii).

\begin{lemma} \label{lem:ker(phi)}
  We have $\ker(\phi)=\Gal(\tilde{L}/L)$. 
\end{lemma}	

\begin{proof}
Since the resolution $\X\to\X_{0,L}$ is defined over $L$, it is clear that $\Gal(\tilde{L}/L)$ is contained in $\ker(\phi)$. 

Fix $\sigma\in\Gamma$ and set $h:=g\cdot\lexp{\sigma}{g}^{-1}$. As $g$ and $\lexp{\sigma}{g}$ are both perfect, Proposition \ref{prop:perfect} implies that $h\in U_t$, i.e.\ we can write
\[
    h = \begin{pmatrix}
    	1 & \Pi^2\delta & \Pi^3\epsilon \\ 0 & 1 & \Pi\zeta \\ 0 & 0 & 1
    \end{pmatrix},
\]
with $\delta,\epsilon,\zeta\in\OO_{\tilde{L}}$. A direct calculation, using Lemma \ref{lem:monodromy_explicit} now shows that 
\begin{equation} \label{eq:phi(sigma)}
   \phi(\sigma) = \begin{pmatrix}
   	 1 & \chi(\sigma)^{-2} \bar{\delta} & \chi(\sigma)^{-3}\bar{\epsilon} \\
   	 0 & \chi(\sigma)^{-2} & \chi(\sigma)^{-3}\bar{\zeta} \\
   	 0 & 0 & \chi(\sigma)^{-3}
   \end{pmatrix}, 
\end{equation}
where $\bar{\delta},\bar{\epsilon},\bar{\zeta}\in k$ are the residue classes of $\delta,\epsilon,\zeta$.

Assume  that $\phi(\sigma)=1$. Then \eqref{eq:phi(sigma)} shows that $\bar{\delta}=\bar{\epsilon}=\bar{\zeta}=0$, i.e\ that $h\in U_t^+$. Since $g$ is assumed to be strongly rigidified, Lemma \ref{lem:unique_hensel_lift} applies, and shows that $\lexp{\sigma}{g}=g$. It follows that $\sigma|_{L_0}=1$ and hence, by \eqref{eq:ker_phi_0} that $\sigma\in\Gal(\tilde{L}/L)$. 

This completes the proof of Lemma \ref{lem:ker(phi)} and hence of Theorem \ref{thm:L_is_minimal_extension}.
\end{proof}

\section{Implementation} \label{sec:implementation}

\subsection{The routine \texttt{resolve\_cusp}} \label{subsec:resolve_cusp}

We have implemented an algorithm which explicitly computes a resolution 
$\X \to \X_0$ of a cusp $P \in \X_{0,s}$ in normal form, thus making Theorem~\ref{thm:cusp_resolution} effective. 
The implementation is part of the Sage package 
\href{https://github.com/kst3rn/StabilityFunction}{\tt StabilityFunction}
(\cite{KletusGitHub}) and contained in the submodule
\href{https://github.com/kst3rn/StabilityFunction/blob/main/semistable_model/curves/cusp_resolution.py}
{\tt cusp\_resolution}. 

Throughout the implementation we work over number fields endowed with a fixed 
$p$-adic valuation, rather than over $p$-adic fields directly. More precisely, 
we restrict to number fields $K$ for which the chosen $p$-adic valuation is 
unique, so that the completion $\widehat K$ coincides with the $p$-adic field 
relevant on the theoretical level. This allows us to perform all computations 
in an exact algebraic setting over number fields, while faithfully modeling 
the $p$-adic situation via completion. 

The theoretical results in the preceding sections were stated under the
simplifying assumption that the residue field is algebraically closed. In the
implementation the residue field of the chosen valuation on a number field is
of course finite. This causes no difficulty: all geometric assertions about
the cusp, the exceptional cubic and semistability are understood after base
change to an algebraic closure of the residue field. If necessary, one first
passes to a finite unramified extension in order to make the cusp and the
required tangent data rational. The subsequent ramified extensions are then
treated exactly as in the theoretical construction.

The central function is called {\tt resolve\_cusp}. It takes as input a 
homogeneous form $F \in K[z,x,y]$ of degree $\geq 3$ over a number field $K$ and a discrete 
valuation $v_K$ on $K$. 
Here $F$ is assumed to be integral and primitive with respect to $v_K$, so that 
it defines an integral plane model $\X_0$ of a smooth plane curve over $K$.\footnote{So that the dehomogenization $f:=F(1,x,y)$ defines the affine plane model used in this article.} 
It is further assumed that the point $P=(1:0:0)$ lies on the special fiber 
$\X_{0,s}$ and represents a cusp in normal form \eqref{eq:cusp_normal_form}.

The output of {\tt resolve\_cusp} is either a tuple $(v_L,t,E,e)$  or 
$(v_L,t,E,T)$, depending on the optional flag {\tt compute\_matrix}. 
Here $v_L$ denotes the unique extension of $v_K$ to a finite extension 
$L/K$, the field over which the rigidified coordinate change is defined, $t\in\QQ$ is the thickness of the resolution, and $E$ is the tail of the resolution (a plane cubic over the residue field in Weierstrass normal form). The integer \(e\) is the additional tame ramification degree needed to adjoin
an element \(\Pi\) with \(v(\Pi)=t\). Thus the actual resolution is defined over the tame extension \(L'/L\) of degree \(e\). If \texttt{compute\_matrix=True}, the function instead returns the explicit  base change matrix $T \in \GL_3(L')$, as in \eqref{eq:T_matrix}, realizing the transformation.

\medskip

Internally, the computation proceeds as follows. 
In order to obtain suitable coordinates, we must determine elements
$\alpha,\beta,\gamma \in \m_L$ such that the form $\lexp{g}{F}$, obtained from $F$ by the coordinate change with the unipotent matrix $g$, as in \eqref{eq:g_matrix}, satisfies Condition \ref{cond:coeffs} (b). We compute $\alpha,\beta,\gamma$ as {\em approximate solutions} of the system of equations \eqref{eq:A_equations} (the exact solutions satisfy the stronger Condition \ref{cond:coeffs} (a)). In other words, we have to find a closed point on $\AA_K^3$ which is $v_K$-adically sufficiently close to a point on $S^{\rm adm}=S^\rig\cap U_0^+$.
Finding such a solution is delegated to the submodule \href{https://github.com/swewers/StabilityFunction/blob/main/semistable_model/curves/approximate_solutions.py}{\tt approximate\_solutions}.

\medskip
Let $J \subset K[\alpha,\beta,\gamma]$ be the ideal defined by the
rigidifying equations \eqref{eq:A_equations}. In computations it is important
not to work with all of $S=\Spec(K[\alpha,\beta,\gamma]/J)$, since $S$ may
contain non-reduced or positive-dimensional components which are irrelevant
for the desired coordinate change. But Lemma \ref{lem:S_et^adm} allows us to replace $S$ by its \'etale part $S_\et$. Concretely, if $\Delta$ denotes the Jacobian determinant of the
three equations \eqref{eq:A_equations}, we replace $J$ by the saturation
\[
J_{\rm et}=J:\Delta^\infty .
\]
This keeps precisely the part of $S$ on which the Jacobian
has full rank and where $S$ is therefore \'etale over $K$. In the situation relevant for the algorithm this yields a reduced
zero-dimensional algebra over $K$.  

After a generic linear change of the variables $\alpha,\beta,\gamma$, the ideal
$J_{\rm et}$ is put into so-called {\em shape position} (\cite{shape_lemma}). Thus the system is represented in the
form
\[
f(\alpha)=0,\qquad \beta=r(\alpha),\qquad \gamma=s(\alpha),
\]
with $f,r,s\in K[\alpha]$. The problem is then reduced to finding a root
$\alpha$ of $f$ in a finite valued extension $L/K$ such that
\[
\alpha,\quad r(\alpha),\quad s(\alpha)\in \mathfrak m_L .
\]
This is done by computing $p$-adic approximations to the irreducible factors of
$f$ over the completion $\widehat K$, using MacLane valuations as implemented
in Sage. Among the resulting branches we choose one satisfying the above
valuation conditions; the existence of such a branch is guaranteed by
Lemma \ref{lem:S_et^adm}.

Finally, the field may be enlarged further in order to contain an element
$\Pi$ with prescribed valuation $v_L(\Pi)=t$. From this data the weighted
blow-up, the exceptional Weierstrass cubic, and, if requested, the base-change
matrix are computed explicitly. The current implementation is designed to
produce a field over which the resolution exists. It does not yet systematically
minimize this field extension, although this could be done by picking those approximate solutions corresponding to a strongly rigidified solution, see Definition \ref{def:strongly_rigidified} and Theorem \ref{thm:L_is_minimal_extension}. In practice, almost all the solutions found by our implementation have this property.

\subsection{Use in stable reduction of plane quartics}
\label{subsec:quartic_reduction}

The cusp-resolution algorithm is one of the local ingredients in the algorithm
for stable reduction of smooth plane quartics described in \cite{SSW}. Starting
from a smooth plane quartic $X/K$, one first computes, after a finite extension
of $K$, a GIT-semistable plane model
\[
\X_0 \subset \PP^2_{\mathcal O_K}.
\]
If $\X_{0,s}$ is strictly GIT-semistable, then the stable reduction is
hyperelliptic and the method does not continue. If $\X_{0,s}$ is GIT-stable,
then its singularities are at worst nodes and cusps. The nodes already occur in
the stable special fiber. Each cusp, however, corresponds to a $1$-tail of the
stable reduction, and the stable model is obtained from $\X_0$ by the local
modifications described in this paper.

Concretely, after putting each cusp of $\X_{0,s}$ into normal form, the global
algorithm calls \texttt{resolve\_cusp}. The routine computes the required local
field extension, the thickness of the attaching node, and the Weierstrass cubic
which becomes the exceptional component. These local resolutions are then
performed independently at the cusps and combined over a common finite
extension of the ground field. In this way the local construction of
Theorem~\ref{thm:cusp_resolution} supplies exactly the missing step between a
GIT-stable plane model and the stable model of the quartic. Examples and Sage
transcripts are given in \cite{SSW}.

\vspace{2ex}\noindent
{\bf Data and code availability statement:}
No datasets were generated or analysed in this article. The implementation of
the cusp-resolution algorithm described in Section~\ref{sec:implementation}
is available in the SageMath package \texttt{StabilityFunction} at
\url{https://github.com/kst3rn/StabilityFunction}. The relevant routine is
\texttt{resolve\_cusp} in the cusp-resolution module. 

\vspace{2ex}\noindent
{\bf Declaration of generative AI and AI-assisted technologies:}
During the preparation of this manuscript, the author used ChatGPT (OpenAI) to assist with language editing, restructuring
of explanatory passages, LaTeX drafting, and drafting or debugging auxiliary
Sage/Python code. All AI-assisted output was reviewed, edited, and independently
checked by the author. The author take full responsibility for the mathematical
content, proofs, computations, code, and conclusions of the article.

\vspace{2ex}\noindent
{\bf Conflict of interest statement:}
The author declares that he has no conflict of interest.

\bibliographystyle{plain}
\bibliography{literature}

\end{document}